\documentclass[preprint,12pt]{elsarticle}    

\usepackage{graphicx,epsfig}
\usepackage{enumerate}
\usepackage{amssymb}
\usepackage{amsthm}
\usepackage{amsmath}
\usepackage{centernot}
\usepackage{xcolor}
\usepackage{times}
\usepackage{mathrsfs}
\usepackage[utf8]{inputenc}
\usepackage{subfigure}
\usepackage{multirow}
\usepackage{xcolor}
\usepackage{tikz}
\usetikzlibrary{positioning}
\usepackage{pgfplots}
\usepackage{pst-plot}
\usepackage{amsmath}
\usepackage{amsfonts}
\usepackage{amssymb}
 \usetikzlibrary{angles,quotes}  

\pgfplotsset{compat=1.10}
\usepgfplotslibrary{fillbetween}
 
\definecolor{mylightblue}{RGB}{153,204,255}

\usepackage{mathtools}
\usepackage{appendix}
\usepackage{wrapfig}
\usepackage{amsmath}
\usepackage{setspace}
\usepackage{booktabs}
\usepackage{lscape}
\usepackage{algorithm,algorithmicx}
\usepackage{algcompatible}
\usepackage{bm}
\usepackage{graphicx,epsfig}
\usepackage{xfrac}
\usepackage{url}
\usepackage{multirow}
\usepackage{longtable}
\usepackage[linkcolor=blue, urlcolor=blue, citecolor=blue,
colorlinks, bookmarks]{hyperref}

\newtheorem{thm}{Theorem}[section]
\theoremstyle{definition}
\newtheorem{dfn}{Definition}[section]
\newtheorem{note}{Note}[]
\newtheorem{lem}{Lemma}[section]
\newtheorem{example}{Example}[section]
\newtheorem{rmrk}{Remark}[]

\makeatletter
\newcommand{\thickhline}{%
	\noalign {\ifnum 0=`}\fi \hrule height 1pt
	\futurelet \reserved@a \@xhline
}

\journal{Information Sciences}
\begin{document}
\begin{frontmatter}
\title{\textbf{Generalized-Hukuhara Subdifferential Analysis and Its Application in  Nonconvex Composite Optimization Problems with Interval-valued Functions}}
		\author[iitbhu_math]{Anshika}
		\ead{anshika.rs.mat19@itbhu.ac.in}
		\author[iitbhu_math]{Debdas Ghosh}
		\ead{debdas.mat@iitbhu.ac.in}
		\author[iitbhu_math]{Ram Surat Chauhan}  
		\ead{rschauhan.rs.mat16@itbhu.ac.in} 
		\author[radko,radko2]{Radko Mesiar\corref{cor1}}
        \ead{mesiar@math.sk}
		\address[iitbhu_math]{Department of Mathematical Sciences,  Indian Institute of Technology (BHU) Varanasi \\ Uttar Pradesh--221005, India}
		\address[radko]{Faculty of Civil Engineering, Slovak University of Technology, Radlinsk\'{e}ho 11, 810 05, Bratislava}
        \address[radko2]{Palack\'y University Olomouc, Faculty of Science, Department of Algebra and Geometry, 17. listopadu 12, 771 46 Olomouc, Czech Republic}		
	    \cortext[cor1]{Corresponding author}

		\begin{abstract}
		In this article, we study  $gH$-subdifferential calculus of convex interval-valued functions (IVFs) and apply it in a nonconvex composite model of interval optimization problems (IOPs). It is found that the $gH$-directional derivative of maximum of finitely many comparable IVFs is the maximum of their $gH$-directional derivative. Proposed concepts of $gH$-subdifferential are observed to be useful to derive  Fritz-John-type and KKT-type efficiency conditions for weak efficient solutions of IOPs. Further, we extract a necessary and sufficient condition to characterize the weak efficient solutions  of nonconvex composite IOPs by applying the proposed concepts. To derive the results on $gH$-subdifferentials, the concepts of limit supremum and limit infimum with certain properties for IVFs are defined in the sequel. The whole analysis is supported by appropriate expository examples. 
		\end{abstract}

		\begin{keyword}
			Interval-valued functions\sep Interval optimization\sep Efficient solution\sep Fritz-John-type condition\sep KKT-type condition \sep Nonconvex composite model.
			\\ \vspace{0.7cm}
		\end{keyword}
		\end{frontmatter}

\section{Introduction}
Nonconvex optimization problems are of great interest while modeling the problems in applied mathematics and operations research. A vast majority of machine learning algorithms train their models by solving optimization problems in which the designed objective is nonconvex. However, due to randomness and imprecision in real-world phenomenon and the uncertainty of the given data, these problems are often modeled by optimization problems whose objective functions is interval-valued. These are represented as interval optimization problems (IOPs). This emerges as an important research topic due to the implicit presence of indefinite and undetermined events in different real-world situations and it is engrossing from the last two decades.
\subsection{Literature survey}
Sunaga \cite{sunaga} has discussed the fundamental of interval analysis to manoeuvre interval uncertainity that emerges in mathematical or computer models by the help of matrix computations using interval arithmetic. However, these results attracted substantial recognition after the advent of the first book by Moore \cite{Moore1966} on interval analysis. It is evident that the interval arithmetic in \cite{Moore1966,Moore1979} has a downside of nonexistence of the additive inverse of a non-degenerate interval (whose upper and lower limits are not equal). That is, for a non-degenerate interval $\textbf{X}$, there is no existence of an interval $\textbf{Y}$ such that $\textbf{X} \oplus \textbf{Y} = \textbf{0}$. To annihilate this difficulty, a new rule for the difference of compact intervals was given by Hukuhara \cite{hukuhara}, known as the Hukuhara difference of intervals ($H$-difference). Although $H$-difference provides the additive inverse of compact intervals, $H$-difference of a compact interval $\textbf{Y}$ from a compact interval $\textbf{X}$ can be evaluated if the width of $\textbf{X}$ is considered to be greater than or equal to that of $\textbf{Y}$.  To resolve these difficulties, Stefanini and Bede \cite{stefanini2009} introduced the generalized Hukuhara difference ($gH$-difference) of intervals. The $gH$-difference not only provides an additive inverse of any compact interval but also applicable for all pairs of compact intervals.\\

To observe the properties of an IVF, calculus plays an essential role. Initially, Hukuhara  \cite{hukuhara} established the concept of differentiability of IVFs with the help of $H$-difference of intervals. However, the definition of Hukuhara differentiability ($H$-differentiability) is found to be restrictive (see \cite{chalco2013}). Despite of that, Wu \cite{wu2007karush} illustrated the idea of limit, continuity, and differentiability, with the help of the Hausdorff metric for intervals. To overcome the deficiency of $H$-differentiability, Markov \cite{markov} introduced the idea of a nonstandard subtraction and developed the calculus of intervals by using this difference. Thereafter, Stefanini and Bede \cite{stefanini2009} refined the concept of $H$-difference for any pair of intervals with the introduction of $gH$-difference, which corresponds with the nonstandard subtraction operator \cite{markov}. Moreover, Chalco-Cano et al. \cite{chalco2013} developed the $gH$-differentiability by using $gH$-difference of intervals for IVFs. Recently, Ghosh et al. \cite{Ghosh2019derivative} have proposed the idea of $gH$-directional derivative for IVFs. Also, in \cite{Ghosh2019gradient} the concept of $gH$-subgradient and $gH$-subdifferential that is equipped with a linearity concept of IVFs is illustrated.\\

With the help of the developed calculus for IVFs, many researchers have developed optimality conditions for IOPs. Wu \cite{wu2007karush} derived KKT relations for IOPs with the help of $H$-difference of intervals. Chalco-Cano et al. \cite{chalco2013} presented KKT optimality results for IOPs, by applying generalized derivatives. In \cite{Ishibuchi1990}, a technique to solve a linear IOP was proposed. Singh et al. \cite{singh} employ  the $gH$-derivative  \cite{wu2007karush} and the partial ordering to intervals \cite{dsingh} to formulate KKT relations for IOPs by considering the sum of lower and upper functions. Ghosh et al. \cite{Ghosh2019kkt} proposed the KKT results for constrained and unconstrained IOPs by observing the geometrical significance of the solutions. Recently, Ghosh \cite{Ghosh2016newton} analyzed the notion of $gH$-differentiability of IVFs and proposed a Newton-type method for IOPs. Many authors have also reported different solution concepts for IOPs; for instance, see \cite{burjee2012, Ghosh2017quasi, Ghosh2018saadle, liu2007, zhou} and the references therein.\\

 Despite many attempts to develop the calculus of nonsmooth IVFs, the existing ideas are not adequate to derive the optimality conditions of nonsmooth IOPs. Although some researchers \cite{jayswal,sunwang} introduced a concept of \emph{LU} optimal solution for nondifferentiable interval-programming problems with invex IVFs, the KKT theory for IOPs in \cite{jayswal, sunwang} are very restrictive. Even very simple IVFs do not follow those results (see Note \ref{luoptimal}). This article thus attempts to derive Fritz-John-type and KKT-type efficiency conditions for nonsmooth IOPs that do not involve any invexity assumption and widely applicable. We also try to derive the weak efficient solution of IOPs that observes the comparability of functions. In the sequel, to observe the role of developed calculus in nonconvex nonsmooth IOPs, we propose a condition by which a weak efficient solution of a nonconvex model can be obtained.

\subsection{Motivation and contribution of the paper}

From the literature of IOPs, it is observed that the optimality condition for nonsmooth IOPs is yet to be developed. However, the concept of subgradients and subdifferential inevitably arise. Thus, in this article, we, at first, briefly propose $gH$-subdifferential calculus for IVFs. The $gH$-subgradient of convex functions is important for nonsmooths IVFs and in dealing with variational problems. We, importantly, report how to calculate the subgradient of the maximum of a finite number of convex IVFs in terms of the subgradient of the IVFs. In the sequel, we derive a Fermat-type, a Fritz-John-type and a KKT-type condition for (weak) efficient solutions of nonsmooth IOPs. The proposed subdifferential calculus can be applied in various areas of nonsmooth IOPs including nonconvex nonsmooth optimization \cite{bagirov,kiwi}, DC (difference of convex functions) programming problems \cite{sun}, etc. In this article, we apply the proposed study on a nonconvex composite IOP.

\subsection{Delineation}
	The proposed work is organised as follows. Section \ref{section2} covers fundamentals of interval arithmetic followed by convexity and calculus of IVFs. In the section \ref{section3}, we define the notions of infimum and supremum for IVF, and the concept of closedness and boundedness for a set of intervals along with the $gH$-directional derivative of maximum of functions. In the same section, we derive a Fritz-John-type condition for IOPs which helps in developing the KKT-type optimality condition for IOPs. After that, the next section \ref{section4} is devoted to the application of $gH$-subdifferentials in obtaining the efficient solutions of a nonconvex composite model in IOPs. Finally the last section is concerned with future directions for our research.


	\section{\textbf{Preliminaries and terminologies}}\label{section2}
	
	\noindent This section is devoted to some basic nomenclature on intervals along with the convexity and calculus of IVFs. In this article, we use the following notations. 
	
	\begin{itemize}
	\item $\mathbb{R}$ represents the set of real numbers
    \item $\mathbb{R}_+$ represents the set of nonnegative real numbers
    \item $I(\mathbb{R})$ represents the set of all closed and bounded intervals
    \item The elements of $I(\mathbb{R})^n$ are  denoted by $\widehat{\textbf X},~\widehat{\textbf Y},~\widehat{\textbf Z},\ldots$
    \item $\textbf{0}$ represents the interval $[0,0]$
    \item $\overline{I(\mathbb{R})}=I(\mathbb{R})\cup \{\boldsymbol{-\infty},\boldsymbol{+\infty}\}$.
    \end{itemize}

	\subsection{Arithmetic of intervals and their dominance relation}\label{ssai}
		In this article, the elements of $I(\mathbb{R})$ is denoted by bold capital letters: ${\textbf X}, {\textbf Y}, {\textbf Z}, \ldots $. An element $\textbf{Y}$ of $I(\mathbb{R})$ in its interval form is represented by the corresponding small letter: $ \textbf{Y} = [\underline{y}, \overline{y}]$, where $\underline{y}~\text{and}~\overline{y}~\text{are real numbers such that}~\underline{y} \leq \overline{y}.$\\
		
Let $\textbf{Y}, \textbf{Z} \in I(\mathbb{R})$ and $\delta \in \mathbb{R}$. Moore's \cite{Moore1966,Moore1979} interval addition, subtraction, product, division and scalar multiplication are denoted by $\textbf{Y} \oplus \textbf{Z},~ \textbf{Y} \ominus \textbf{Z},~ \textbf{Y} \odot \textbf{Z},~ \textbf{Y} \oslash \textbf{Z}$, and $\delta \odot \textbf{Y}$, respectively. In defining $\textbf{Y} \oslash \textbf{Z}$, it is assumed that $0\notin \textbf{Z}$. \\

Since $\textbf{X}\ominus\textbf{X}\neq\textbf{0}$ for any nondegenerate interval $\textbf{X}$, we use the following concept of difference of intervals in this article.

	\begin{dfn}($gH$-difference of intervals \cite{stefanini2009}). The $gH$-difference of $\textbf{Y}$ and $\textbf{Z}\in I(\mathbb{R})$ is defined as the interval $\textbf{C}$, and denoted as $\textbf{Y}\ominus_{gH}\textbf{Z}$ such that
		\[\textbf{Y} =  \textbf{Z} \oplus  \textbf{C}~\text{or}~\textbf{Z} = \textbf{Y} \ominus \textbf{C}.\]
	
		For intervals $\textbf{Y} = \left[\underline{y},~\overline{y}\right]$ and $\textbf{Z} = \left[\underline{z},~\overline{z}\right]$,
		\[
		\textbf{Y} \ominus_{gH} \textbf{Z} = \left[\min\{\underline{y}-\underline{z},
		\overline{y} - \overline{z}\},~ \max\{\underline{y}-\underline{z}, \overline{y} -
		\overline{z}\}\right] \text{ and } \textbf{Y} \ominus_{gH} \textbf{Y} = \textbf{0}.
		\]
	\end{dfn}

	The algebraic operations on the product space $I(\mathbb{R})^n=I(\mathbb{R})\times I(\mathbb{R})\times \cdots  \times I(\mathbb{R})$ ($n$ times) are defined as follows.
	
	\begin{dfn} (Algebraic operations on $I(\mathbb{R})^n$ \cite{Ghosh2019gradient}). For two elements $\widehat{\textbf{Y}}=(\textbf{Y}_1,\textbf{Y}_2,\ldots,\textbf{Y}_n)$ and $\widehat{\textbf{Z}}=(\textbf{Z}_1,\textbf{Z}_2,\ldots,\textbf{Z}_n)$ of $I(\mathbb{R})^n$, an algebraic operation $\boldsymbol{\star}$ between $\widehat{\textbf{Y}}$ and $\widehat{\textbf{Z}}$, denoted by $\widehat{\textbf{Y}} \boldsymbol{\star}\widehat{\textbf{Z}}$, defined by
	\[
	\widehat{\textbf{Y}}\boldsymbol{\star}\widehat{\textbf{Z}}=(\textbf{Y}_1\boldsymbol{\star}\textbf{Z}_1,\textbf{Y}_2\boldsymbol{\star}\textbf{Z}_2,\ldots,\textbf{Y}_n\boldsymbol{\star}\textbf{Z}_n),
	\]
	where $\boldsymbol{\star} \in \{\oplus,\ominus,\ominus_{gH}\}$.
	
	\end{dfn}

		In the following, we provide a domination relation on intervals that is used throughout the paper. We remark that \emph{domination} in the following definition is  based on a \emph{minimization} type optimization problems: a \emph{smaller value} is  \emph{better}.

\begin{dfn}(Dominance of intervals \cite{wu2007karush}). Consider two intervals $\textbf{Y}$ and $\textbf{Z}$ in $I(\mathbb{R})$.
	 	\begin{enumerate}[(i)]
	 		\item $\textbf{Y}$ is called \emph{dominated} by $\textbf{Z}$ if `$\underline{z}~\leq~ \underline{y}$ and $\overline{z}~\leq~\overline{y}$', and then we write $\textbf{Z}~\preceq~ \textbf{Y}$;
	 		\item $\textbf{Y}$ is said to be \emph{strictly dominated} by $\textbf{Z}$ if either `$\underline{z} ~\leq~ \underline{y}$  and $\overline{z} ~<~ \overline{y}$' or `$\underline{z} ~<~ \underline{y}$  and $\overline{z} ~\leq~ \overline{y}$', and then we write $\textbf{Z}~\prec~ \textbf{Y}$;
	 		\item if $\textbf{Y}$ is not dominated by $\textbf{Z}$, then we write $\textbf{Z}~\npreceq~ \textbf{Y}$; if $\textbf{Y}$ is not strictly dominated by $\textbf{Z}$, then we write $\textbf{Z}~\nprec~ \textbf{Y}$;
	 		\item if $\textbf{Z}~\npreceq~ \textbf{Y}$ and $\textbf{Y}~\npreceq~ \textbf{Z}$, then it will be said that \emph{none of} $\textbf{Z}$ and $\textbf{Y}$ \emph{dominates the other}, or $\textbf{Z}$ and $\textbf{Y}$ \emph{are not comparable};
	 		\item if $`\textbf{Z}~\preceq~ \textbf{Y}$' or $`\textbf{Y}~\preceq~ \textbf{Z}$', then it will be said that $\textbf{Z}$ and $\textbf{Y}$ \emph{are comparable}.
	 	\end{enumerate}
	 \end{dfn}
	 
	 \begin{rmrk}
	 For any two elements $\widehat{\textbf{Y}}=(\textbf{Y}_1,\textbf{Y}_2,\ldots,\textbf{Y}_n)$ and $\widehat{\textbf{Z}}=(\textbf{Z}_1,\textbf{Z}_2,\ldots,\textbf{Z}_n)$ in $I(\mathbb{R})^n$,
	 \[\widehat{\textbf{Y}} \preceq \widehat{\textbf{Z}} \iff \textbf{Y}_j \preceq \textbf{Z}_j~\text{for all}~j=1, 2, \ldots, n.\]

\begin{dfn}(Set of comparable intervals). Let $\textbf{S}\subseteq I(\mathbb{R}).$ Then, $\textbf{S}$ is said to be a \emph{set of comparable intervals} if for every $\textbf{X},\textbf{Y}\in \textbf{S}$, either $\textbf{X}$ dominates $\textbf{Y}$ or $\textbf{Y}$ dominates $\textbf{X}$.

\end{dfn}

\end{rmrk}


	\subsection{Convexity and differential calculus of IVFs}
	A function $\textbf{T}$ from a nonempty subset $\mathcal{Y}$ of $\mathbb{R}^n$ to $I(\mathbb{R})$ is known as an IVF. For each $y \in \mathcal{Y}$,~ $\textbf{T}$ is presented by
	\[
	\textbf{T}(y)=\left[\underline{t}(y),\overline{t}(y)\right],
	\]
    where $\underline{t}$ and $\overline{t}$ are real-valued functions on $\mathcal{Y}$ such that $\underline{t}(y)\leq \bar{t}(y)$ for all $y \in \mathcal{Y}$.\\
    If $\mathcal{Y}$ is convex, then the IVF
		$\textbf{T}$ is said to be \emph{convex} on $\mathcal{Y}$ \cite{wu2007karush} if for any $y_1$ and $y_2 \in \mathcal{Y},\delta_1,~\delta_2\in[0,1]\text{ with }\delta_1+\delta_2=1$,
		\[
		\textbf{T}(\delta_1 y_1+\delta_2 y_2)\preceq
		\delta_1\odot\textbf{T}(y_1)\oplus\delta_2\odot\textbf{T}(y_2).\]

\noindent
	The IVF $\textbf{T}$ is said to be \emph{$gH$-continuous} \cite{Ghosh2016newton} at $\bar{y} \in \mathcal{Y}$ if
		\[
		\lim_{\lVert h \rVert\rightarrow 0}\left(\textbf{T}(\bar{y}+h)\ominus_{gH}\textbf{T}(\bar{y})\right)=\textbf{0}.\]
		 If $\textbf{T}$ is $gH$-continuous at each $y$ in $\mathcal{Y}$, then $\textbf{T}$ is said to be \emph{$gH$-continuous} on $\mathcal{Y}$.

	\begin{lem}(See \label{lc1}\cite{wu2007karush}).
		If an IVF $\textbf{T}$ is convex on a convex set $\mathcal{Y} \subseteq \mathbb{R}^n$, then $\underline{t}$
		and $\overline{t}$ are convex on $\mathcal{Y}$ and vice-versa.
	\end{lem}

	\begin{lem}(See \cite{Ghosh2019gradient}).
		If an IVF $\textbf{T}$ is $gH$-continuous on subset $\mathcal{Y}$ of $\mathbb{R}^n$, then
		$\underline{t}$ and $\overline{t}$ are continuous on $\mathcal{Y}$ and vice-versa.
	\end{lem}
	
\begin{dfn}(Proper IVF). Let $\textbf{T}:\mathcal{Y}\to \overline{I(\mathbb{R})}$ be an extended IVF on a nonempty subset $\mathcal{Y}$ of $ \mathbb{R}^n$. Then, $\textbf{T}$ is called \emph{proper} if there exists a $\bar y\in\mathcal{Y}$ such that 
\[
\textbf{T}(\bar y)\prec \boldsymbol{+\infty} \text{ and } \boldsymbol{-\infty}\prec\textbf{T}(y)~\text{for all}~ y\in\mathcal{Y}.
\]
\end{dfn}

\begin{dfn}(Domain of an IVF). 	Let $\mathcal{Y}\subseteq \mathbb{R}^n$. For an extended IVF $\textbf{T}:\mathcal{Y}\to \overline{I(\mathbb{R})}$, the \emph{domain} of $\textbf{T}$, denoted as dom \textbf{T}, is defined as
\[\text{dom}~\textbf{T}=\{y \in \mathcal{Y}:\textbf{T}(y)\prec\boldsymbol{+\infty}\}.
\]

\end{dfn}

	\begin{dfn}($gH$-derivative \cite{chalco2012}).
	Let $\mathcal{Y}\subseteq I(\mathbb{R})$, then for an IVF $\textbf{T}: \mathcal{Y} \to I(\mathbb{R})$, the \emph{$gH$-derivative} at $\bar{y} \in \mathcal{Y}$ is defined by
	\[
	\textbf{T}'(\bar{y})=\underset{d \rightarrow 0}{\lim} \frac{\textbf{T}(\bar{y}+d)\ominus_{gH} \textbf{T}(\bar{y})}{d},~\text{provided the limit exists.}\]
	\end{dfn}
	

	\begin{dfn} ($gH$-partial derivative \cite{Ghosh2016newton}). Let $\textbf{T}: \mathcal{Y} \to I(\mathbb{R})$ be an IVF on a subset $\mathcal{Y}$ of $\mathbb{R}^n$. We define a function $\textbf{H}_j$ by
	\[
	\textbf{H}_j(y_j) = \textbf{T}(\bar{y}_1,\bar{y}_2,\ldots,\bar{y}_{j-1},y_j,\bar{y}_{j+1},\ldots,\bar{y}_n),\]
	where $\bar{y}=(\bar{y}_1,\bar{y}_2,\ldots,\bar{y}_n)^\top \in \mathcal{Y}.$ Let the generalized derivative of $\textbf{H}_j$ exists at $\bar{y}_j$, then the $j$-th \emph{$gH$-partial derivative} of \textbf{T} at $\bar{y},$ denoted as $D_j\textbf{T}(\bar{y})$, is defined as
	\[
	D_j\textbf{T}(\bar{y})= \textbf{H}'_j(\bar{y}_j)~\text{for all}~ j=1,2,\ldots,n.\]
	\end{dfn}
	

	\begin{dfn}($gH$-gradient \cite{Ghosh2016newton}).
	Let $\mathcal{Y}$ be a subset of $\mathbb{R}^n$, then the \emph{$gH$-gradient} of an IVF $\textbf{T}$ at a point $\bar{y} \in \mathcal{Y}$ is defined by the interval vector
	\[
	\nabla\textbf{T}(\bar{y})=(D_1\textbf{T}(\bar{y}),D_2\textbf{T}(\bar{y}),\ldots,D_n\textbf{T}(\bar{y}))^\top.\]
	The $gH$-gradient is denoted by $\nabla\textbf{T}(\bar{y})$.
	\end{dfn}

	
		\begin{dfn}\label{ddd}
		($gH$-directional derivative \cite{Ghosh2019derivative}).
		Let $\textbf{T}$ be an IVF on a subset $\mathcal{Y}$ of $\mathbb{R}^n$. Let $\bar{y} \in \mathcal{Y}$ and $h \in \mathbb{R}^n$ and
		\[
		\lim_{\delta \to 0+}\frac{1}{\delta}\odot\left(\textbf{T}(\bar{y}+\delta h)\ominus_{gH}\textbf{T}(\bar{y})\right)\text{ exists finitely. }
		\]
		 Then, the limit at $\bar{y}$ in the direction $h$ is said to be \emph{$gH$-directional derivative} of $\textbf{T}$, and it is denoted by $\textbf{T}_\mathscr{D}(\bar{y})(h)$.
	\end{dfn}

Next, we explore the relation of efficient solution of the IOP 
\begin{align}\label{arars}
		\displaystyle \min_{y \in \mathcal{Y}} \textbf{T}(y)
		,
		\end{align}
		where $\textbf{T}: \mathcal{Y} \rightarrow I(\mathbb{R})$ be an IVF on the nonempty subset $\mathcal{Y}$ of $\mathbb{R}^n$.

\begin{dfn}\label{efficient_point_def}(Efficient point \cite{Ghosh2019derivative}).
		A point $\bar{y}\in \mathcal{Y}$ is said to be an \emph{efficient point} of the IOP (\ref{arars})
		if $\textbf{T}(y)\nprec\textbf{T}(\bar{y})$ for all $y\in \mathcal{Y}.$
	\end{dfn}

\begin{dfn}(Weak efficient point). 
	A point $\bar{y}\in \mathcal{Y}$ is said to be a \emph{weak efficient point} of the IOP (\ref{arars})
		if $\textbf{T}(\bar{y})\preceq\textbf{T}(y)$ for all $y\in \mathcal{Y}.$
\end{dfn}


\subsection{Few properties of the elements in  $I(\mathbb{R})$ and $I(\mathbb{R})^n$}
  
 \begin{dfn} \label{norm} (Norm on $I(\mathbb{R})$ \cite{Moore1966}). For an interval $\textbf{Y} = \left[\underline{y}, \bar{y}\right]$ in $ I(\mathbb{R})$, the function ${\lVert . \rVert}_{I(\mathbb{R})}$ from $I(\mathbb{R})$ to $\mathbb{R}_+$, defined by
\[
{\lVert \textbf{Y} \rVert}_{I(\mathbb{R})} = \max \{|\underline{y}|, |\bar{y}|\}, 	
\]
is a \emph{norm} on $I(\mathbb{R})$. 
\end{dfn}

	\begin{dfn} \label{norm_2}
(Norm on $I(\mathbb{R})^n$ \cite{Moore1979}). For an element $\widehat{\textbf{Y}} = (\textbf{Y}_1,\textbf{Y}_2,\ldots,\textbf{Y}_n)$ in $I(\mathbb{R})^n$, the function ${\lVert . \rVert}_{I(\mathbb{R})^n}$ from $I(\mathbb{R})^n$ to $\mathbb{R}_+$, defined by
\[
{\lVert \widehat{\textbf{Y}} \rVert}_{I(\mathbb{R})^n} = \sum_{j=1}^{n} {\lVert \textbf{Y}_j \rVert}_{I(\mathbb{R})}, 
\]	
is a \emph{norm} on $I(\mathbb{R})^n$. 
\end{dfn}

\begin{dfn}\label{max} 
	(Maximum and minimum of intervals). Let $\textbf{Z}_1,\textbf{Z}_2,\ldots,\textbf{Z}_p$ be the elements of $I(\mathbb{R})$ with $\textbf{Z}_1\preceq \textbf{Z}_2 \preceq\cdots\preceq \textbf{Z}_p$. Then, 
	\[\max\{\textbf{Z}_1, \textbf{Z}_2 ,\ldots, \textbf{Z}_p\} = \textbf{Z}_p~\text{and}~\min\{\textbf{Z}_1, \textbf{Z}_2, \ldots, \textbf{Z}_p\} = \textbf{Z}_1.\]
	
\end{dfn}

\begin{rmrk}
It can be easily observe that the maximum of any finite set $\textbf{S}\subseteq I(\mathbb{R})$ of comparable intervals lies in the set $\textbf{S} $.
\end{rmrk}


\begin{lem}\label{00}
For intervals $\textbf{X}, \textbf{Y}, \textbf{Z}~\text{and}~\textbf{W}~\text{of}~ I(\mathbb{R})$,
\begin{enumerate}[(i)]
    \item \label{3_3} if $\textbf{X} \oplus\textbf{Y} \preceq \textbf{Z}\oplus\textbf{W}$, then $\textbf{X}\ominus_{gH} \textbf{Z}\preceq\textbf{W}\ominus_{gH} \textbf{Y} $;
    \item \label{3_5} if $\textbf{X}\preceq\textbf{Y}~\text{and}~\textbf{Y}\preceq\textbf{Z}$, then $ \textbf{P}\preceq\textbf{R}.$
\end{enumerate}
 \end{lem}

 \begin{proof}
     See \ref{appendix_inq}.
 \end{proof}
%

\begin{rmrk}
Let $\textbf{X}, \textbf{Y}, \textbf{Z}$ and $\textbf{W}$ of $I(\mathbb{R})$. If ($\textbf{X}\oplus \textbf{Y})\ominus_{gH}(\textbf{Z}\oplus \textbf{W}) = (\textbf{X} \ominus_{gH}\textbf{Z})\ominus_{gH}(\textbf{W} \ominus_{gH}\textbf{Y})$, then (\ref{3_3}) of  Lemma \ref{00} is an obvious property. However, $(\textbf{X}\oplus \textbf{Y})\ominus_{gH}(\textbf{Z}\oplus \textbf{W})$ is not always equal to $\big(\textbf{X}\ominus_{gH}\textbf{Z}\big) \ominus_{gH} \big(\textbf{W}\ominus_{gH}\textbf{Y}\big)$. For instance, consider $\textbf{X} = [-3, 2],~\textbf{Y} = [0,0],~\textbf{Z} = [4, 10]$, and $\textbf{W} = [-7.5, -6]$, then
    \begin{eqnarray*}
   &&(\textbf{X}\oplus \textbf{Y})\ominus_{gH}(\textbf{Z}\oplus \textbf{W})=[-2,0.5]
    ~\text{and}~(\textbf{X} \ominus_{gH}\textbf{Z})\ominus_{gH}(\textbf{W} \ominus_{gH}\textbf{Y})=[-1,-0.5].
        \end{eqnarray*}
     Therefore, (\ref{3_3}) of  Lemma \ref{00} is not an obvious property.
\end{rmrk}

\begin{note}\label{notei}
For any $\textbf{X},\textbf{Y}, \textbf{Z}$ and $\textbf{W}$ of $I(\mathbb{R})$, if we consider $\textbf{Y}=\textbf{0}$ in (\ref{3_3}) of Lemma \ref{00}, then $\textbf{X} \preceq \textbf{Z}\oplus \textbf{W}  \implies \textbf{X}\ominus_{gH} \textbf{Z}\preceq\textbf{W}$.
 \end{note}

\begin{note}\label{noteii}
For any $\textbf{X},\textbf{Y}$ and $\textbf{Z}$ of $I(\mathbb{R})$ if $\textbf{X} \preceq \textbf{Y}$, then $\textbf{X}\oplus \textbf{Z} \preceq \textbf{Y}\oplus \textbf{Z}$. Thus, from (\ref{3_3}) of Lemma \ref{00}, we obtain
    \[\textbf{X}\ominus_{gH} \textbf{Z} \preceq \textbf{Y}\ominus_{gH} \textbf{Z}~\text{ or }~ \textbf{Z}\ominus_{gH}\textbf{Y} \preceq \textbf{Z}\ominus_{gH}\textbf{X}.\]
\end{note}


	\begin{dfn}(Infimum of a set of intervals \label{lbound} \cite{gourav2020}). Let $\textbf{S}\subseteq \overline{I(\mathbb{R})}$. An interval $\mathbf{{X}}\in I(\mathbb{R})$ is said to be a lower bound of $\textbf{S}$ if 
	\[
	\mathbf{{X}}\preceq \textbf{Y}~\text{for all}~ \textbf{Y}\in \textbf{S}.\] 
	A lower bound $\mathbf{{X}}$ of $\textbf{S}$ is called an \emph{infimum} of $\textbf{S}$, denoted by $\inf\textbf{S}$, if
	\[
	\textbf{Z}\preceq \mathbf{{X}}~\text{for all lower bounds}~ \textbf{Z}~\text{of}~ \textbf{S}~\text{in}~ I(\mathbb{R}).
	\] 

	\end{dfn}


	\begin{dfn}(Supremum of a set of intervals \label{ubound}\cite{gourav2020}). Let $\textbf{S}\subseteq \overline{I(\mathbb{R})}$. An interval $\textbf{{X}}\in I(\mathbb{R})$ is said to be an upper bound of $\textbf{S}$ if 
	\[
	\textbf{Y}\preceq \mathbf{{X}}~\text{ for all}~ \textbf{Y} \in \textbf{S}.\] An upper bound $\mathbf{{X}}$ of $\textbf{S}$ is called a \emph{supremum} of $\textbf{S}$, denoted by $\sup\textbf{S}$, if 
	\[
	\mathbf{{X}}\preceq \textbf{Z}~\text{for all upper bounds}~\textbf{Z}~\text{of}~\textbf{S}~\text{in}~ I(\mathbb{R}).\]

	\end{dfn}


	\begin{rmrk}(See \cite{gourav2020}).
		Let $\textbf{S}=\left\{[a_\mu,b_\mu]\in \overline{I(\mathbb{R})}: \mu\in\Lambda~ \text{and}~ \Lambda~\text{being an index set}~\right\}$. Then, by definition \ref{lbound} and \ref{ubound}, it follows that $\inf\textbf{S}=\left[\inf\limits_{\mu\in\Lambda}a_\mu,~\inf\limits_{\mu\in\Lambda}b_\mu\right]$ and $\sup\textbf{S}=\left[\sup\limits_{\mu\in\Lambda}a_\mu,~\sup\limits_{\mu\in\Lambda}b_\mu\right].$
	\end{rmrk}

\begin{rmrk}  
  Let $\textbf{S}\subseteq$ $I(\mathbb{R})$ be a finite set of comparable intervals, then infimum and supremum of $\textbf{S}$ coincide with minimum and maximum of the set $\textbf{S}$, respectively. 
   
\end{rmrk}	
	


\section{$gH$-subdifferential calculus for interval-valued functions}\label{section3}

In this section, we derive some results based on $gH$-subgradient and $gH$-subdifferential of IVFs. In the sequel, the role of the proposed Fermat-type condition in developing a Fitz-John-type and a KKT-type condition for IOPs is shown.

\begin{dfn}{(Convex set of intervals).} 
A nonempty subset $\textbf{S}$ of $I(\mathbb{R})^{n}$ is said to be a \emph{convex set of intervals} if for every $\widehat{\textbf{Y}},~\widehat{\textbf{Z}} \in \textbf{S}$,
\[
\delta_{1}\odot\widehat{\textbf{Y}}\oplus\delta_{2}\odot\widehat{\textbf{Z}} \in \textbf{S}~\text{for all }\delta_{1},~\delta_{2} \in [0,1]~\text{with}~\delta_{1}+\delta_{2}=1.\]
\end{dfn}


\begin{dfn}{(Convex combination of intervals).}
Let $\widehat{\textbf{X}}$ be an interval in $I(\mathbb{R})^{n}$. Then,  $\widehat{\textbf{X}}$ is said to be a convex combination of the intervals  $\widehat{\textbf{X}}_{1},\widehat{\textbf{X}}_{2},\ldots,\widehat{\textbf{X}}_{p} \in I(\mathbb{R})^{n}$ if 
\[
\widehat{\textbf{X}}=\bigoplus_{j=1}^{p} \delta_{j}\odot \widehat{\textbf{X}}_{j}~\text{with}~\delta_{j} \geq 0 ~\text{and}~\sum_{i=j}^{p} \delta_{j}=1.
\]
\end{dfn}


\begin{dfn}{(Convex hull of a set of intervals).} \label{cohull}
For a nonempty set $\textbf{S}\subseteq I(\mathbb{R})^{n}$, the \emph{convex hull} of  set $\textbf{S}$, co($\textbf{S}$), is defined by 
\[
\text{co(\textbf{S})}=\left\{\widehat{\textbf{X}}\in I(\mathbb{R})^{n} : \widehat{\textbf{X}}=\bigoplus_{j=1}^{p} \delta_{j}\odot \widehat{\textbf{X}}_{j},~\widehat{\textbf{X}}_{j}\in \textbf{S}~\text{with}~\delta_{j}\geq 0~\text{and}~\sum_{j=1}^{p} \delta_{j}=1\right\}.\]
\end{dfn}


	\begin{dfn}\label{sup}(Supremum of an IVF). Let $\textbf{S}\subseteq I(\mathbb{R})^n$ and $\textbf{T}: \textbf{S} \rightarrow \overline{I(\mathbb{R})}$ be an extended IVF. Then, the \emph{supremum} of $\textbf{T}$, is defined as
	\[
	\sup_{\textbf{S}}\textbf{T}=\sup\{\textbf{T}(\widehat{\textbf{X}}): \widehat{\textbf{X}}\in \textbf{S}\}.
	\]
	\end{dfn}


		\begin{dfn}\label{inf}(Infimum of an IVF). Let $\textbf{S} \subseteq I(\mathbb{R})^n$ and $\textbf{T}: \textbf{S} \rightarrow \overline{I(\mathbb{R})}$ be an extended IVF. Then, the \emph{infimum} of $\textbf{T}$, is defined as 
	\[
	\inf_{\textbf{S}}\textbf{T}=\inf\{\textbf{T}(\widehat{\textbf{X}}): \widehat{\textbf{X}}\in \textbf{S}\}.
	\]
	\end{dfn}

	%

	
\begin{lem}\label{infsup}
Let $\textbf{S}\subseteq I(\mathbb{R})^{n}$ and $\textbf{T} : \textbf{S} \rightarrow \overline{I(\mathbb{R})}$ be an extended IVF. Then, for $\textbf{S}_1, \textbf{S}_2\subseteq \textbf{S}$ with $\textbf{S}_1 \subseteq \textbf{S}_2$ and $\delta \geq 0,$
\begin{enumerate}[(i)]
\item\label{4_1} $\underset{ \textbf{S}_2}{\inf}~\textbf{T}\preceq\underset{ \textbf{S}_1}{\inf}~\textbf{T}$, \\

\item \label{4_2}$\underset{\textbf{S}_{1}}{\sup}~\textbf{T}\preceq\underset{\textbf{S}_{2}}{\sup}~\textbf{T}$,\\

 \item\label{4_3} $\underset{\textbf{S}}{\inf}~(\delta\odot\textbf{T})=\delta \odot \underset{\textbf{S}}{\inf}~\textbf{T}$, and \\
 
 \item\label{4_4} $\underset{\textbf{S}}{\sup}~(\delta\odot\textbf{T})=\delta \odot\underset{\textbf{S}}{\sup}~\textbf{T}$.
\end{enumerate}
\end{lem}

\begin{proof}
See \ref{appendix_insu}.
\end{proof}


\begin{lem}\label{supinf}
Let $\textbf{S} \subseteq I(\mathbb{R})^{n}$ and $\textbf{T}_1 ,~\textbf{T}_2 : \textbf{S} \rightarrow \overline{I(\mathbb{R})}$ be extended IVFs. Then,

\begin{enumerate}[(i)]
\item\label{5_1} $\underset{\textbf{S}}{\inf}~\textbf{T}_1 \oplus \underset{\textbf{S}}{\inf}~\textbf{T}_2 \preceq \underset{\textbf{S}}{\inf}~(\textbf{T}_1 \oplus\textbf{T}_2)$ and \\

\item\label{5_2}  $ \underset{\textbf{S}}{\sup}~(\textbf{T}_1 \oplus \textbf{T}_2)\preceq\underset{\textbf{S}}{\sup}~\textbf{T}_1 \oplus \underset{\textbf{S}}{\sup}~\textbf{T}_2$.
\end{enumerate}
\end{lem}

\begin{proof}
See \ref{appendix_suprinfm}.
\end{proof}

	\begin{dfn}{(Sequence in $I(\mathbb{R})^n$.} An IVF $\widehat{\textbf{T}}:\mathbb{N}\rightarrow I(\mathbb{R})^n$ is called a \emph{sequence} in $I(\mathbb{R})^n$.
 	\end{dfn}


\begin{dfn}\label{g22}{(Convergence of a sequence in $I(\mathbb{R})^n$).} 
		A sequence $\{\widehat{\textbf{G}}_k\}$ in $I(\mathbb{R})^n$ is said to be \emph{convergent} to $\widehat{\textbf{G}} \in I(\mathbb{R})^n$ if for each $\epsilon>0$, there exists an $m\in \mathbb{N}$ such that
		\[
		\lVert \widehat{\textbf{G}}_k \ominus_{gH} \widehat{\textbf{G}} \rVert_{I(\mathbb{R})^n} < \epsilon~\text{for all}~k\geq m.\]
		The interval $\widehat{\textbf{G}}$ is called limit of the sequence $\{\widehat{\textbf{G}}_k\}$ and is presented by $\underset{k \to \infty}{\lim}{\widehat{\textbf{G}}}_k= \widehat{\textbf{G}}$.
		 
\end{dfn}


\begin{rmrk}
It is to be acclaimed that if a sequence $\{\widehat{\textbf{G}}_k\}$ in $I(\mathbb{R})^n$ converges to some $\widehat{\textbf{G}} \in I(\mathbb{R})^n$, where $\widehat{\textbf{G}}_k=(\textbf{G}_{1_k},\textbf{G}_{2_k},\ldots,\textbf{G}_{n_k})$ and $\widehat{\textbf{G}}=(\textbf{G}_1,\textbf{G}_2,\ldots,\textbf{G}_n)$, then from Definitions \ref{norm} and \ref{norm_2}, the sequence $\{\textbf{G}_{j_k}\}$ in $ I(\mathbb{R})$ converges to $\textbf{G}_j \in  I(\mathbb{R})$ for each $j=1,2,\ldots,n$.
\end{rmrk}

\begin{lem} \label{compdd}
Let $\{\textbf{X}_k\}$ and $ \{\textbf{Y}_k\}$ be two sequence in $I(\mathbb{R})^n$ and $\underset{k \to \infty}{\lim} \textbf{X}_k=\textbf{X}$ and $\underset{k \to \infty}{\lim} \textbf{Y}_k=\textbf{Y}$. If $\textbf{X}_k \preceq \textbf{Y}_k$ for all $k$, then $\textbf{X} \preceq \textbf{Y}.$
\end{lem}

\begin{proof}
From $\textbf{X}_k \preceq \textbf{Y}_k $, we have
\begin{eqnarray*}
&& \underline{x}_k \leq \underline{y}_k~\text{and}~\overline{x}_k \leq \overline{y}_k\\
&\implies& \underset{k \to \infty}{\lim} \underline{x}_k \leq \underset{k \to \infty}{\lim} \underline{y}_k~\text{and}~\underset{k \to \infty}{\lim}\overline{x}_k \leq \underset{k \to \infty}{\lim}\overline{y}_k \\
&\implies& \underline{x} \leq \underline{y}~\text{and}~\overline{x} \leq \overline{y}\\
&\implies&\textbf{X} \preceq \textbf{Y}.
\end{eqnarray*}
\end{proof}


\begin{lem} \label{subsequence}
Let  $\{\widehat{\textbf{G}}_k\}$ be a sequence in $I(\mathbb{R})^n$ that converges to an interval $\widehat{\textbf{G}} \in I(\mathbb{R})^n$. Then, every subsequence of $\{\widehat{\textbf{G}}_k\}$ converges to the same limit $\widehat{\textbf{G}}$.
\end{lem}

\begin{proof}
Let $\{\widehat{\textbf{G}}_{k_j}\}$ be a subsequence of the convergent sequence $\{\widehat{\textbf{G}}_k\}$. Since $\{\widehat{\textbf{G}}_k\}$ has a limit $\widehat{\textbf{G}}$, for each $\epsilon>0$, there exists an $m \in \mathbb{N}$ such that 
	\begin{eqnarray*}
	&&\lVert \widehat{\textbf{G}}_k \ominus_{gH} \widehat{\textbf{G}} \rVert_{I(\mathbb{R})^n} < \epsilon~\text{for all}~k\geq m.
		\end{eqnarray*}
As $\{k_j\}$ is an increasing sequence of natural numbers, there exists $p \in \mathbb{N}$ such that $k_j \geq m$ for each $j \geq p.$ Therefore,
			\begin{eqnarray*}
	\lVert \widehat{\textbf{G}}_{k_j} \ominus_{gH}  \widehat{\textbf{G}} \rVert_{I(\mathbb{R})^n} < \epsilon~\text{for all}~j\geq p.
	\end{eqnarray*}
	This implies that $\{\widehat{\textbf{G}}_{k_j}\}$ converges to $\widehat{\textbf{G}}$.
\end{proof}


\begin{dfn}{(Closed set in $I(\mathbb{R})^n)$.} \label{closedset}
A nonempty subset $\textbf{S} \subseteq I(\mathbb{R})^n$ is said to be \emph{closed} iff for every convergent sequence $\{\widehat{\textbf{G}}_k\}$ in $\textbf{S}$ converging to $ \widehat{\textbf{G}}$, $ \widehat{\textbf{G}}$ must belong to $\textbf{S}$.

\end{dfn}


\begin{dfn} {(Bounded set in $I(\mathbb{R})^n)$.} \label{boundedset}
Let $\textbf{S} \subseteq I(\mathbb{R})^n$. An interval $\widehat{{{\textbf{A}}}} \in I(\mathbb{R})^n$ is said to be a lower bound of $\textbf{S}$ if 
\[\widehat{{\textbf{A}}}\preceq \widehat{\textbf{C}}~\text{for all}~\widehat{\textbf{C}} \in \textbf{S}.\] 
An interval $\widehat{{\textbf{B}}} \in I(\mathbb{R})^n$ is said to be an upper bound of $\textbf{S}$ if 
\[
\widehat{\textbf{C}} \preceq \widehat{{\textbf{B}}}~\text{for all}~\widehat{\textbf{C}} \in \textbf{S}.\]
A nonempty subset $\textbf{S} \subseteq I(\mathbb{R})^n$ which is bounded above and bounded below is said to be \emph{bounded}.

\end{dfn}


\begin{thm}{\emph{(Finite union of closed sets in $I(\mathbb{R})^n$).}} \label{closedunion}
Let $\textbf{S}_1,\textbf{S}_2,\ldots,\textbf{S}_p$ be a  finite collection of closed sets in $I(\mathbb{R})^n$. Then, $\bigcup_{j=1}^{p} \textbf{S}_j$ is closed.
\end{thm}

\begin{proof}
Let $\{\widehat{\textbf{G}}_k\}$ be an arbitrary sequence in $\bigcup_{j=1}^{p} \textbf{S}_j$ that converges to $\widehat{\textbf{G}} \in I(\mathbb{R})^n$, where $\widehat{\textbf{G}}_k=(\textbf{G}_{1_k},\textbf{G}_{2_k},\ldots,\textbf{G}_{n_k})$ and $\widehat{\textbf{G}}=(\textbf{G}_1,\textbf{G}_2,\ldots,\textbf{G}_n)$. 
Since $\{\widehat{\textbf{G}}_k\}$ contains infinitely many terms and $\bigcup_{j=1}^{p}\textbf{S}_j$ is union of finite number of sets, there exists at least one $\textbf{S}_j$ that contains infinitely many terms of the sequence $\{\widehat{\textbf{G}}_k\}$. Hence, by Lemma \ref{subsequence}, we get a subsequence of $\{\widehat{\textbf{G}}_k\}$ in $\textbf{S}_j$ which converges to  $\widehat{\textbf{G}}$ .
Since $\textbf{S}_j$ is closed, by Definition \ref{closedset}, $\widehat{\textbf{G}} \in \textbf{S}_j$,  which implies that $\widehat{\textbf{G}} \in \bigcup_{j=1}^{p} \textbf{S}_j$. 
Hence, $\bigcup_{j=1}^{p} \textbf{S}_j$ is closed.
\end{proof}



\begin{thm} {\emph{(Finite union of bounded sets in $I(\mathbb{R})^n$).}} \label{finitebounded}
Let $\textbf{S}_1,\textbf{S}_2,\ldots,\textbf{S}_p$ be a finite collection of bounded sets in $I(\mathbb{R})^n$. Then, $\bigcup_{j=1}^{p} \textbf{S}_j$ is bounded.

\end{thm}
\begin{proof}
Since $\textbf{S}_1,\textbf{S}_2,\ldots,\textbf{S}_p$ are bounded sets in $I(\mathbb{R})^n$, there exist $\widehat{{\textbf{X}}}_1,\widehat{{\textbf{X}}}_2,\ldots,$ $\widehat{{\textbf{X}}}_p $ and\\ $\widehat{{\textbf{Y}}}_1,\widehat{{\textbf{Y}}}_2,\ldots,\widehat{{\textbf{Y}}}_p$ in $I(\mathbb{R})^n$ such that 
\begin{eqnarray}\label{dvn_1}
\widehat{{\textbf{X}}}_j \preceq \widehat{\textbf{C}}_j \preceq \widehat{{\textbf{Y}}}_j~\text{for each}~\widehat{\textbf{C}}_j~\text{in}~ \textbf{S}_j,~j= 1, 2, \ldots, p.
\end{eqnarray}
where $\widehat{\textbf{X}}_j=(\textbf{X}_{1_j},\textbf{X}_{2_j},\ldots,\textbf{X}_{n_j})$ and $\widehat{\textbf{Y}}_j=(\textbf{Y}_{1_j},\textbf{Y}_{2_j},\ldots,\textbf{Y}_{n_j})$.\\  
Let $\widehat{\textbf{X}}=(\textbf{X}_1,\textbf{X}_2,\ldots,\textbf{X}_n)$ and $\widehat{\textbf{Y}}=(\textbf{Y}_1,\textbf{Y}_2,\ldots,\textbf{Y}_n)$, where $\textbf{X}_\mu=[\min\underline{x}_{\mu_j},\min\overline{x}_{\mu_j}]$ and $ \textbf{Y}_\mu=[\max\underline{y}_{\mu_j},\max\overline{y}_{\mu_j}]$ for all $\mu=1, 2, \ldots, n$ and $j=1, 2, \ldots, p$. Then,
\begin{eqnarray}\label{dvn_2}
\widehat{\textbf{X}}\preceq \widehat{\textbf{X}}_j~\text{and}~\widehat{\textbf{Y}}_j \preceq \widehat{\textbf{Y}}~\text{for each}~j=1, 2, \ldots, p.
\end{eqnarray}
From (\ref{dvn_1}) and (\ref{dvn_2}), we obtain that for all $j=1,2,\ldots,p$, 
%
\begin{align*}
&\widehat{\textbf{X}} \preceq \widehat{\textbf{C}}_j \preceq \widehat{\textbf{Y}}\\
\text{or},~& \widehat{\textbf{X}} \preceq \widehat{\textbf{Z}} \preceq \widehat{\textbf{Y}},~\text{where}~\widehat{\textbf{Z}}\in\bigcup_{j=1}^{p}  \textbf{S}_j.
\end{align*} 
Hence, $\bigcup_{j=1}^{p} \textbf{S}_j $ is a bounded set.
\end{proof}


\begin{thm}{\emph{(Closedness of the convex hull of a set in $I(\mathbb{R})^n$).}} \label{closedhull}
Let $\textbf{S}$ be a nonempty closed set in $I(\mathbb{R})^n$. Then, the convex hull of $\textbf{S}$ is a closed set.
\end{thm}

\begin{proof}
Let $\{\widehat{\textbf{G}}_k\}$ be an arbitrary sequence in $\text{co(\textbf{S})}$ that converges to $\widehat{\textbf{G}} \in I(\mathbb{R})^n$, where $\widehat{\textbf{G}}_k=(\textbf{G}_{1_k},\textbf{G}_{2_k},\ldots,\textbf{G}_{n_k})$ and $\widehat{\textbf{G}}=(\textbf{G}_1,\textbf{G}_2,\ldots,\textbf{G}_n)$.\\
Since $\{\widehat{\textbf{G}}_k\} \in~\text{co(\textbf{S})}$,  there exists $\{\delta_{j}^k\} \subset \mathbb{R}_+$ and $\{\widehat{\textbf{G}}_j^{k}\} \in \textbf{S}$ such that 
\[
\widehat{\textbf{G}}_k= \bigoplus_{j=1}^{n}~\delta_{j}^{k}\odot\widehat{\textbf{G}}_j^{k}~\text{for all}~k \in \mathbb{N}~\text{with}~\sum_{j=1}^{n} \delta_{j}^{k}=1.\]
Also, $\delta_{j}^{k} \geq$ 0 with $\sum_{j=1}^{n}  \delta_{j}^{k}=1$ gives that $ \delta_{j}^{k}$ is a bounded sequence. Thus, we have 
\[
\delta_j^{k} \to \delta_j~\text{ such that }~\delta_j \geq 0~\text{ with }\sum_{j=1}^{n} \delta_j=1.
\]
Since $\textbf{S}$ is closed,  $\{\widehat{\textbf{G}}_j^{k}\} \in \textbf{S}$ must have a convergent subsequence. Assume that $\widehat{\textbf{G}}_j^{k} \to \widehat{\textbf{G}}_j$. Therefore, by closedness of $\textbf{S}$, 
\[ 
\widehat{\textbf{G}}_j \in \textbf{S}~\text{and}~ 
 \widehat{\textbf{G}}= \bigoplus_{j=1}^{n}\delta_j \odot \widehat{\textbf{G}}_j \in \text{co(\textbf{S})}.\]
Thus, $\widehat{\textbf{G}}_k \to \widehat{\textbf{G}}\in \text{co(\textbf{S})}.$ Hence, $\text{co(\textbf{S})}$ is a closed set.
\end{proof}


%

\begin{thm}{\emph{(Boundedness of convex hull of a set in $I(\mathbb{R})^n$).}} \label{boundedhull}
Let $\textbf{S}$ be a nonempty bounded set in $I(\mathbb{R})^n$. Then, the convex hull of $\textbf{S}$ is bounded.
\end{thm}

\begin{proof}
Since $\textbf{S}$ is bounded, there exist $\widehat{\textbf{X}}$ and $\widehat{\textbf{Y}}$ in $I(\mathbb{R})^n$ such that
\begin{align*}
&\widehat{\textbf{X}} \preceq \widehat{\textbf{C}}_j \preceq \widehat{\textbf{Y}}~\text{for all}~\widehat{\textbf{C}}_j~\text{in}~ \textbf{S}.
\end{align*}
Therefore, for all $\widehat{\textbf{C}}_j$ in $\textbf{S}$, $\delta_j \geq 0,$ and $\sum_{j=1}^{p} \delta_{j}=1$, we obtain 
\begin{align*}
&\bigoplus_{j=1}^{p}\delta_j \odot\widehat{\textbf{X}} ~\preceq~ \bigoplus_{j=1}^{p}\delta_j \odot \widehat{\textbf{C}}_j~\preceq~ \bigoplus_{j=1}^{p}\delta_j \odot \widehat{\textbf{Y}}\\
\implies&\bigoplus_{j=1}^{p}\delta_j \odot\widehat{\textbf{X}} ~\preceq ~\widehat{\textbf{Z}}~\preceq~\bigoplus_{j=1}^{p}\delta_j \odot \widehat{\textbf{Y}}\\
\implies&\widehat{\textbf{X}}~\preceq~\widehat{\textbf{Z}}~\preceq~\widehat{\textbf{Y}},
\end{align*}
where $\widehat{\textbf{Z}}=\bigoplus_{j=1}^{p}\delta_j \odot \widehat{\textbf{C}}_j \in \text{co(\textbf{S})}$. Since $\widehat{\textbf{Z}}$ is arbitrary, $\text{co(\textbf{S})}$ is a bounded set.
\end{proof}




\begin{dfn}{($gH$-subgradient \cite{Ghosh2020lasso}).}
Let $\textbf{T} : \mathcal{Y}\rightarrow I(\mathbb{R})$ be a convex IVF on a convex subset $\mathcal{Y}$ of $\mathbb{R}^{n}$. Then, an element $\widehat{\textbf{G}}=(\textbf{G}_{1},\textbf{G}_{2},\ldots,\textbf{G}_{n}) \in I(\mathbb{R})^{n}$ is said to be a \emph{$gH$-subgradient} of $\textbf{T}$ at $\bar{y}$ if
\[
(y-\bar{y})^\top\odot \widehat{\textbf{G}}~\preceq~\textbf{T}(y)\ominus_{gH}\textbf{T}(\bar{y}).
\]
The collection of all subgradients of $\textbf{T}$ at $\bar{y}$ is called $gH$-subdifferential and is denoted by $ \partial\textbf{T}(\bar{y})$.
\end{dfn}

%
%

\begin{thm}(See \cite{Ghosh2020lasso}).\label{diffder}
Let $\mathcal{Y}$ be a nonempty subset of $\mathbb{R}^n$ and $\textbf{T}:\mathcal{Y}\to I(\mathbb{R})$ be $gH$-differentiable IVF at $\bar{y}\in \mathcal{X}$. Then, $\textbf{T}$ has $gH$-directional derivative at $\bar{y}$ for every direction $h \in \mathbb{R}^n$ and
\[\partial\textbf{T}(\bar{y})=\{\nabla \textbf{T}(\bar{y})\}.
\]
\end{thm}

\begin{example} \label{rrr}
Let $\mathcal{Y}$ be a convex subset of $\mathbb{R}^n$ and an IVF $\textbf{T}: \mathcal{Y} \rightarrow I(\mathbb{R})$ be defined by 
\begin{eqnarray}\label{crs_1}
\textbf{T}(x) = \lVert y \rVert_{0} \odot \textbf{C},
\end{eqnarray} 
where $\lVert y \rVert_{0}$ is the number of nonzero components in $y$ and $\textbf{0} \prec \textbf{C}\in I(\mathbb{R})$. Then, $gH$-subgradient of $\textbf{T}$ at $\bar{y}$ $\in \mathcal{X}$ is obtain by the following cases: \\
\begin{enumerate}[$\bullet$ \textbf{Case} 1.]
    \item \label{c1} If $\bar{y} = 0$ and $\widehat{\textbf{G}} \in \partial \textbf{T}(\bar{y})$, then for all $y\in \mathbb{R}^n$
    
     \[ (y-\bar{y})^\top \odot \widehat{\textbf{G}} \preceq \textbf{T}(y) \ominus_{gH} \textbf{T}(\bar{y})
      \implies y^\top \odot \widehat{\textbf{G}} \preceq  \lVert y\rVert_{0} \odot \textbf{C}.\]
    
    It is possible only for $\widehat{\textbf{G}}=\textbf{0}$. Therefore, $\partial \textbf{T}(0) = \{ \textbf{0}\}$. \\

   \item \label{c2} If $\bar{y} \neq 0$ and $\widehat{\textbf{G}} \in \partial \textbf{T}(\bar{y})$, then
   \begin{eqnarray} \label{mod_1}
      (y-\bar{y})^\top \odot\widehat{\textbf{G}} \preceq \lVert y\rVert_{0} \odot \textbf{C} \ominus_{gH} \lVert \bar{y}\rVert_{0} \odot\textbf{C}~\text{for all}~y\in \mathbb{R}^n.
   \end{eqnarray}
  Let $ \lVert y\rVert_{0}=p$ and $\lVert \bar{y}\rVert_{0}=q$, where $p$ and $q$ are natural numbers such that $p, q \leq n$. Then, we have
   \begin{eqnarray} \label{mod_2}
   \bigoplus_{i=1}^{n}(y_i-\bar{y}_i)\odot \textbf{G}_i \preceq [\min\{(p-q)\underline{c},~(p-q)\overline{c}\},~\max\{(p-q)\underline{c},~(p-q)\overline{c}\}].
   \end{eqnarray}
 Without loss of generality, let the first $m$ components of $(y-x)$ be nonnegative and rest of  $(n-m)$ number of components be nonpositive. Then, from (\ref{mod_2}), 
 \begin{align*}
 &\bigoplus_{i=1}^{m}(y_i-\bar{y}_i)\odot \textbf{G}_i  \bigoplus_{j=m+1}^{n}(y_j-\bar{y}_j)\odot \textbf{G}_j \\
 \preceq~& [\min\{(p-q)\underline{c},~(p-q)\overline{c}\},~
 \max\{(p-q)\underline{c},~(p-q)\overline{c}\}].
 \end{align*}
 Therefore, we get
 \begin{eqnarray*}
 &&\sum_{i=1}^{m}(y_i-\bar{y}_i) \underline{g_i} + \sum_{j=m+1}^{n}(y_j-\bar{y}_j) \overline{g_j} \leq \min\{(p-q)\underline{c},~(p-q)\overline{c}\}~\text{and}\\
  &&\sum_{i=1}^{m}(y_i-\bar{y}_i) \overline{g_i} + \sum_{j=m+1}^{n}(y_j-\bar{y}_j) \underline{g_j} \leq \max\{(p-q)\underline{c},~(p-q)\overline{c}\}.
 \end{eqnarray*}
 The above inequalities hold only when $\underline{g_i},\overline{g_i},\underline{g_j}$ and $\overline{g_j}$ are all zeros and $q \leq p$. But $q \leq p$ for all $y$ holds only when $\bar{y}=0$. Hence, we have a contradiction to $\bar{y}\neq 0$.
 \end{enumerate}
 From Case \ref{c1} and Case \ref{c2}, it is clear that $gH$-subgradient of IVF (\ref{crs_1}) exists only at $\bar{y}$ = 0 and $\partial \textbf{T}(0) = \{ \textbf{0}\}. $
\end{example}


\begin{note}
  The IVF (\ref{crs_1}) of Example \ref{rrr} does not satisfy the property $\textbf{T}(\delta y)=\delta \odot \textbf{T}(Y)$. For instance, consider $y=(1,0,\ldots,0)$ and $\delta=\frac{1}{2}$. Then,
 \begin{align*}
 &\textbf{T}(\delta y)=\textbf{C}~
 \text{and}~\delta\odot\textbf{T}(y) = \frac{1}{2}\odot \textbf{C}.
  \end{align*}
  Hence, $\textbf{T}(\lambda y) \neq \lambda \odot \textbf{T}(y)$. 
 \end{note}

 \begin{note}  
  The IVF (\ref{crs_1}) of Example \ref{rrr} is  not convex. For instance, consider $y_{1}=(1,0,\ldots,0)$ and $y_{2}=(0,1,\ldots,0)$. Then, for $\delta_{1}= \delta_{2}=\frac{1}{2}$, we have
 \[ \delta_1\odot\textbf{T}(y_1)\oplus\delta_2\odot\textbf{T}(y_2) = \textbf{C} \prec 2\odot\textbf{C} = \textbf{F}(\delta_1 y_1+\delta_2 y_2). \]
 Thus, the IVF $\textbf{T}$ is not convex.
\end{note}


\begin{lem}
Let $\mathcal{Y}$ be a subset of $\mathbb{R}^{n}$ and $\textbf{T}:\mathcal{Y}\rightarrow I(\mathbb{R})$ be an IVF. Then, for any $\bar{y}\in\text{dom}(\textbf{T})$ and $\delta \geq$ 0
\[
\partial(\delta\odot\textbf{T})(\bar{y})=\delta\odot\partial\textbf{T}(\bar{y}),\]
where $\text{dom}(\delta\odot\textbf{T})=\text{dom}(\textbf{T})$.
\end{lem}

\begin{proof}
Let $\widehat{\textbf{G}} \in \partial\textbf{T}(\bar{y})$. Then, for any $y\in\text{dom}(\textbf{T})$, 
\begin{eqnarray*}
&&(y-\bar{y})^\top\odot \widehat{\textbf{G}}\preceq\textbf{T}(y)\ominus_{gH}\textbf{T}(\bar{y})\\
&\iff&\delta\odot((y-\bar{y})^\top\odot \widehat{\textbf{G}})\preceq\delta\odot\left(\textbf{T}(y)\ominus_{gH}\textbf{T}(\bar{y})\right)\text{ for } \delta \geq 0\\
&\iff& (y-\bar{y})^\top\odot(\delta\odot \widehat{\textbf{G}})\preceq\delta\odot(\textbf{T}(y)\ominus_{gH}\textbf{T}(\bar{y}))\\
&\iff& (y-\bar{y})^\top\odot(\delta\odot \widehat{\textbf{G}})\preceq(\delta\odot\textbf{T})(y)\ominus_{gH}(\delta\odot\textbf{T})(\bar{y}))\\
&\iff&\delta\odot\widehat{\textbf{G}}\in\partial(\delta\odot\textbf{T})(\bar{y}).
\end{eqnarray*}
\end{proof}

\begin{lem}\label{sum}
Let $\textbf{T}_1,\textbf{T}_2:\mathcal{Y}\to I(\mathbb{R})$ be two convex IVFs on the nonempty convex set $\mathcal{Y}\subseteq \mathbb{R}^n $. Then, $\textbf{T}_1\oplus\textbf{T}_2$ is a convex IVF on $\mathcal{Y}.$
\end{lem}

\begin{proof}
Let $\textbf{T}(y)=\textbf{T}_1(y)\oplus \textbf{T}_2(y)$. Then, for any $y_1,y_2 \in \mathcal{Y}$ and $\delta\in [0,1]$,
\begin{eqnarray*}
\textbf{T}(\delta y_1+(1-\delta) y_2)&=&\textbf{T}_1(\delta y_1+(1-\delta) y_2)\oplus \textbf{T}_2(\delta y_1+(1-\delta) y_2)\\
&\preceq& \delta \odot \textbf{T}_1(y_1)\oplus (1-\delta)\odot\textbf{T}_1(y_2)\oplus \delta \odot \textbf{T}_2(y_1) \oplus (1-\delta)\odot \textbf{T}_2(y_2)\\
&=&\delta \odot (\textbf{T}_1(y_1)\oplus\textbf{T}_2(y_1))\oplus (1-\delta)\odot(\textbf{T}_1(y_2)\oplus\textbf{T}_2(y_2))\\
&=&\delta \odot \textbf{T}(y_1)\oplus (1-\delta)\odot\textbf{T}(y_2).
\end{eqnarray*}
Thus, $\textbf{T}_1\oplus\textbf{T}_2$ is a convex IVF on $\mathcal{Y}$.
\end{proof}

\begin{thm} {\emph{($gH$-directional derivative of the maximum  function).}} \label{dd_max}
Let $\mathcal{Y}$ be a nonempty convex subset of $\mathbb{R}^n$. Let $A$ be any finite set of indices. For each $i \in A,$ let $\textbf{G}_{i} : \mathcal{Y} \rightarrow \overline{I(\mathbb{R})}$ be a convex and $gH$-continuous IVF such that $\textbf{G}_{i _\mathscr{D}}(\bar{y};d)$ exists for all $\bar{y} \in \mathcal{Y}$. Let for each $y\in \mathcal{Y}$, the set $\{\textbf{G}_i(y):i \in A\}$ is a set of comparable intervals and define
\[
\textbf{G}(y) = \underset{i \in A}{\max}~\textbf{G}_{i}(y).\]
Then, for any $\bar{y} \in \mathcal{Y}$ and $d \in \mathcal{Y}$,
\[
\textbf{G}_{\mathscr{D}}(\bar{y})(d) = \underset{i \in I(\bar{y})}{\max}~ {\textbf{G}} _{i _\mathscr{D}}(\bar{y};d),~\text{where}~I(\bar{y}) = \{ i\in A: \textbf{G}_{i}(\bar{y}) = \textbf{G}(\bar{y})\}.\]
\end{thm}

\begin{proof}

Let $\bar{y}\in\mathcal{Y}$ and $d \in\mathcal{Y}$ such that $\bar{y}+ \delta d \in\mathcal{Y}$ for $\delta >$ 0. Then,
\begin{align} \label{0_1}
 &\textbf{G} _{i}(\bar{y} + \delta d) \preceq \textbf{G}(\bar{y} + \delta d) ~\text{for all}~i \in A \nonumber\\
\text{or,}~& \textbf{G} _{i}(\bar{y} + \delta d)\ominus _{gH} \textbf{G} (\bar{y}) \preceq\textbf{G}(\bar{y} + \delta d) \ominus _{gH} \textbf{G} (\bar{y})~\text{by Note~\ref{noteii}, for all}~i\in A \nonumber\\
\text{or,}~&\textbf{G} _{i}(\bar{y} + \delta d) \ominus _{gH} \textbf{G} _{i} (\bar{y}) \preceq \textbf{G}(\bar{y} + \delta d) \ominus _{gH} \textbf{G} (\bar{y})~\text{for all}~i\in I(\bar{y}) \nonumber \\
\text{or,}~& \lim_{\delta \to 0+} \frac{1}{\delta} \odot (\textbf{G} _{i}(\bar{y} + \delta d) \ominus _{gH} \textbf{G} _{i}(\bar{y})) \preceq \lim_{\delta \to 0+} \frac{1}{\delta} \odot (\textbf{G} (\bar{y} + \delta d) \ominus _{gH} \textbf{G} (\bar{y}))~\text{for all}~i\in I(\bar{y}) \nonumber \\
\text{or,}~&\max\textbf{G}_{i_\mathscr{D}}(\bar{y};d) \preceq \textbf{G}_{\mathscr{D}}(\bar{y})(d)~\text{by Lemma \ref{compdd}, for all}~i\in I(\bar{y}).
\end{align}
To prove the converse, we assert that there exists a neighbourhood $\mathcal{N}(\bar{y})$ such that $I(y) \subset I(\bar{y})~\text{for all}~ y \in \mathcal{N}(\bar{y}). $ Assume contrarily that there exists a sequence $\{y_{k}\}$ in $\mathcal{Y}$ with $y_{k} \rightarrow \bar{y}$ such that $I(y_{k}) \not\subset I(\bar{y})$. Choose $i_{k} \in I(y_{k})$ but $i_{k} \notin I(\bar{y})$. Since $I(y_{k})$ is closed, $i_k \to \bar{i} \in I(y_{k})$. By $gH$-continuity of $\textbf{G}_i$, we have
       \begin{eqnarray*}
        \textbf{G} _{\bar{i}}(y_{k}) = \textbf{G} (y_{k}) \implies \textbf{G} _{\bar{i}}(\bar{y}) = \textbf{G}(\bar{y}),
       \end{eqnarray*}
which is a contradiction to $i_{k} \notin I(\bar{y})$. Thus, $I(y) \subset I(\bar{y})$ for all $y \in \mathcal{N}(\bar{y})$.\\
Let us consider \{$\delta _{k}$\} $\subset \mathbb{R}_{+}$, $\delta_{k} \rightarrow$ 0 and $\bar{y} + \delta _{k} d \in \mathcal{N}(\bar{y})$ for all $d \in \mathcal{Y}$. 
Then,
       \begin{align} \label{0_2}
        &  \textbf{G} _{i}(\bar{y})\preceq\textbf{G}(\bar{y}) ~\text{for all}~i \in A \nonumber \\
        \text{or,}~& \textbf{G} (\bar{y} + \delta _{k} d) \ominus _{gH} \textbf{G} (\bar{y}) \preceq  \textbf{G} (\bar{y} + \delta _{k} d) \ominus _{gH} \textbf{G} _{i}(\bar{x})~\text{by Note~\ref{noteii}, for all}~i \in A \nonumber  \\
        \text{or,}~& \textbf{G} (\bar{y} + \delta _{k} d) \ominus _{gH} \textbf{G} (\bar{x})  \preceq \textbf{G} _{i}(\bar{y} + \delta _{k} d) \ominus _{gH} \textbf{G} _{i}(\bar{y})~\text{for all}~i\in I(\bar{y}+\delta_k d) \nonumber \\
        \text{or,}~& \underset{k \rightarrow \infty}{\lim} \frac{1}{\delta_{k}} \odot (\textbf{G} (\bar{y} + \delta _{k} d) \ominus _{gH} \textbf{G} (\bar{y})) \preceq \underset{k \rightarrow \infty}{\lim} \frac{1}{\delta _{k}} \odot (\textbf{G} _{i}(\bar{y} + \delta _{k} d) \ominus _{gH} \textbf{G} _{i}(\bar{y}))~\text{for all}~i\in I(\bar{y}) \nonumber \\
        \text{or,}~& \textbf{G}_{\mathscr{D}}(\bar{y})(d) \preceq  \max~ \textbf{G} _{i_\mathscr{D}}(\bar{y};d)~\text{by Lemma \ref{compdd}, for all}~i\in I(\bar{y}).
      \end{align}
From (\ref{0_1}) and (\ref{0_2}), we obtain
\[
 \textbf{G} _{\mathscr{D}}(\bar{y})(d) = \underset{i\in I(\bar{y})}{\max} \textbf{G} _{i_\mathscr{D}}(\bar{y};d).
\]
\end{proof}

\begin{thm}{\emph{(Fermat-type efficient point condition).}}\label{fermets}
Let $\textbf{T}: \mathcal{Y} \to I(\mathbb{R})\cup \{+\boldsymbol{\infty}\}$ be a convex IVF on a nonempty convex subset $\mathcal{Y}$ of  $\mathbb{R}^n$. Then, $\bar{y}$ is a weak efficient solution of 
$
\underset{y \in \mathcal{Y}}{\inf}~\textbf{F}(y) 
$
if and only if $\widehat{\textbf{0}} \in \partial \textbf{T}(\bar{y})$.
\end{thm}

\begin{proof}
If $\bar{y}$ is a weak efficient solution, then for all $y \in \mathcal{Y}$
\begin{eqnarray*}
 &&\textbf{0} \preceq \textbf{T}(y)\ominus_{gH}\textbf{T}(\bar{y})\\
&\iff&(y-\bar{y})^\top \odot \widehat{\textbf{0}}\preceq \textbf{T}(y)\ominus_{gH}\textbf{T}(\bar{y})\\
&\iff&\widehat{\textbf{0}} \in \partial \textbf{T}(\bar{y}).
\end{eqnarray*}

\end{proof}

\begin{lem} \label{supkk}
Consider the IVF $\textbf{T}: \mathcal{Y} \to I(\mathbb{R})$ and $g_1,g_2,\ldots,g_p : \mathcal{Y} \to \mathbb{R}$. Let for the IOP
\begin{eqnarray}\label{a_1}
&\underset{y \in \mathcal{Y}}{\inf}&\textbf{T}(y)\\
&\text{subject to }&g_j(y) \leq 0,~j=1,2,\ldots,p,\nonumber
\end{eqnarray}
 the infimum value, denoted by $\textbf{T}_{\inf}$, is finite. Consider another IOP which is defined by
 \begin{eqnarray}\label{z1}
\underset{y \in \mathcal{Y}}{\inf}~\textbf{T}^{*}(y)
\end{eqnarray}
where $\textbf{T}^{*}: \mathcal{Y} \to I(\mathbb{R})$ and $\textbf{T}^*(y) = \sup\{\textbf{T}(y)\ominus_{gH} \textbf{T}_{\inf}, g_1(y),g_2(y),\ldots,g_p(y)\}.$ 
Then, the set of weak efficient points of (\ref{a_1}) is same as that of (\ref{z1}).
\end{lem}

\begin{proof}
Let $Y'$ be the set of weak efficient points of (\ref{a_1}). Then, to prove the required relation, we show that
\begin{enumerate}[(i)]
    \item \label{a_11} $ \textbf{T}^{*}(y)=\textbf{0}$ for any $y \in Y'$ or
    \item \label{a_22} $ \textbf{0}\prec \textbf{T}^*(y) $ for any $y \not\in Y'$.
\end{enumerate}
 If $y \in Y'$, then $g_j(y) \leq 0$ for all $i=1,2,\ldots,p$ and $\textbf{T}(y)=\textbf{T}_{\inf}$. This implies that $\textbf{T}^{*}(y)=0$.\\
If $y \not\in Y'$, then $y$ is either `not feasible' or `feasible but not weak efficient'.\\ 
Let $y$ be not feasible. Then, for some $j$, $ g_j(y) > 0$, which implies $\textbf{0}\prec \textbf{T}^*(x)$.\\
Let $y$ be feasible but not weak efficient. Then, $\textbf{T}_{\inf} \prec\textbf{T}(y)$, which implies that $\textbf{0}\prec \textbf{T}^*(y)$.
\end{proof}

\begin{thm}\label{fri}{\emph{(Fritz-John-type necessary condition for IOPs).}}
Let $\mathcal{Y}$ be a nonempty convex subset of $\mathbb{R}^n$. Consider the constrained IOP
\begin{eqnarray} \label{a_33}
&\underset{y \in \mathcal{Y}}{\inf}&\textbf{T}(y)\\
&\text{subject to }& g_j(y) \leq 0,~j=1,2,\ldots,p,\nonumber
\end{eqnarray}
where $\textbf{T}: \mathcal{Y} \to I(\mathbb{R})$ be a convex IVF and $g_1,g_2,\ldots,g_p : \mathcal{Y} \to \mathbb{R}$ be real-valued convex functions. If  $\bar{y}$ is a weak efficient solution of (\ref{a_33}),  
then there exist $\delta_j \geq 0,~j=0,1,2,\ldots,p$, not all zeros, such that 
\begin{eqnarray}\label{an}
&&\widehat{\textbf{0}} \in \delta_0\odot \partial \textbf{T}(\bar{y}) \oplus \sum_{j=1}^{p} \delta_j \partial g_j(\bar{y})\\
&\text{and}&\delta_j g_j(\bar{y})=0,~j=1,2,\ldots,p.\nonumber
\end{eqnarray}
\end{thm}

\begin{proof}
Let $\bar{y}$ be a weak efficient solution of (\ref{a_33}). Denote $\textbf{T}(\bar{y})=\textbf{T}_{\inf}$. Then, by Lemma \ref{supkk}, $\bar{y}$ is a weak efficient solution of
$
\underset{y \in \mathcal{Y}}{\inf}~\textbf{T}^{*}(y),
$
\begin{eqnarray*} 
\text{where}~\textbf{T}^{*}(y)= \sup \{\textbf{T}(y)\ominus_{gH} \textbf{T}_{\inf}, g_1(y),g_2(y),\ldots,g_p(y)\}.
\end{eqnarray*}
Since $\textbf{T}^{*}(\bar{y})=\textbf{0}$, by Theorem \ref{fermets}, 
\begin{eqnarray}\label{a44}
\widehat{\textbf{0}} \in \partial\textbf{T}^{*}(\bar{y}).
\end{eqnarray}
From Theorem 3.50 of \cite{amir}, we have
\begin{eqnarray}\label{a55}
\partial \textbf{T}^{*}(\bar{y})=\delta_0 \odot \partial (\textbf{T}(\bar{y})\ominus_{gH} \textbf{T}_{\inf})\oplus \underset{j \in I(\bar{y})}{\sum} \delta_j \partial g_j(\bar{y}),
\end{eqnarray}
where $\delta_j \geq 0$ such that $\underset{j \in I(\bar{y})}{\sum}\delta_j=1$ and $I(\bar{y})=\{j \in \{1,2,\ldots,p\}: g_j(\bar{y})=0\}$.\\
In view of (\ref{a44}) and (\ref{a55}), we obtain
\[\widehat{\textbf{0}} \in \delta_0 \odot \partial(\textbf{T}(\bar{y})\ominus_{gH} \textbf{T}_{\inf})\oplus \underset{j \in I(\bar{y})}{\sum} \delta_j \partial g_j(\bar{y}).
\]
Let $\textbf{H}(\bar{y})=\textbf{T}(\bar{y})\ominus_{gH} \textbf{T}_{\inf}.$ Then, $\partial(\textbf{H})(\bar{y})=\partial \textbf{T}(\bar{y})$. Now, by taking $\delta_j=0$ for $j \not\in I(\bar{y})$, we get the desired result.
\end{proof}

\begin{note}\label{frit}
If $\textbf{T}(\bar{y})$ and $\textbf{T}(y)$ are not comparable, where $\bar{y}$ is an efficient solution of IOP (\ref{a_33}), then (\ref{an}) is not true at $\bar{y}$. Consider the IOP
\begin{align}\label{iop}
\inf&~\textbf{T}(y) = [1,2]\odot y^2 \oplus [0,2]\odot y \oplus [2,5]\\
\text{subject to }&~g(y)= y-1 \leq 0,~y \in \mathcal{Y}=[-2,0]. \nonumber
\end{align} Note that $\textbf{T}$ (refer to Figure \ref{fig1}) and $g$ are convex on $\mathcal{Y}$. Also, \textbf{T} is $gH$-differentiable on $\mathcal{Y}$ as $\underline{t}(y)=y^2+2y+2$ and $\overline{t}(y)=2y^2+5$ are differentiable on $\mathcal{Y}$. Thus, from Theorem \ref{diffder},
\[
\partial \textbf{T}(y)=\{\nabla\textbf{T}(y)\}=\{ [2,4]\odot y \oplus [0,2]\}~\text{and }\partial g(y)=\{\nabla g(y)\}=\{1\}~\text{for all}~y \in \mathcal{Y}.
\]

\begin{figure}[H]
	\begin{center}
\begin{tikzpicture}[scale=0.5]
\begin{axis}[
    samples=30,
    domain=-2:0,
    axis lines=middle,
    y=3cm/5,
    x=7cm,
    xtick={-2, -1.5, ...,0},
    minor xtick = {-1.9, -1.8, ..., 0}, 
    ytick={0, 4,...,12},
    minor ytick = {1, 3, ..., 13},
    enlargelimits = true, 
    xlabel=$y$,
    ylabel=$T$, no markers]
    
    
    \addplot+[name path=A,black] {x^2 + 2*x + 2}; 
    \addplot+[name path=B,black] {2*x^2 + 5};
    \draw (60, 12) node {$\underline{t}$};
    \draw (100, 70) node {$\overline{t}$};
    \addplot[mylightblue] fill between[of=A and B];
    \draw[color=red, ultra thick]  (100, -10) -- (200, -10); 
\end{axis}
 
\end{tikzpicture}
\caption{The IVF $\textbf{T}$ of Note  \ref{frit}}\label{fig1}

\end{center}
\end{figure}

All $y$'s in the red line segment in Figure \ref{fig1} are efficient solutions of \eqref{iop}. Thus, $\bar{y}=0 \in [-1, 0]$ is an efficient solution of the IOP (\ref{iop}). However, for all $\delta_0, \delta_1\geq 0$, not all zeros,
\begin{eqnarray*}
\delta_{0}\odot\partial\textbf{T}(\bar{y})\oplus \delta_1 \partial g(\bar{y}) =\delta_{0}\odot[0, 2]\oplus \delta_1 \odot [1,1] \neq \textbf{0}.
\end{eqnarray*}
Hence, $\textbf{0}\not\in \delta_{0}\odot\partial\textbf{T}(\bar{y})\oplus \delta_1 \partial g(\bar{y}).$
\end{note}

\begin{note}\label{luoptimal}
It is to observe that the IOP in Note \ref{frit} also violets the KKT theory for IOPs in \cite{jayswal}. Since
\[\underline{\delta_{0}}\partial\underline{t}(\bar{y})+\overline{\delta_{0}}\partial\overline{t}(\bar{y})+ \delta_1 \partial g(\bar{y})=2\underline{\delta_{0}}+\delta_1 \neq 0, \]
where $\underline{\delta_{0}}>0,\delta_1 \geq 0$, therefore $0 \not\in \underline{\delta_{0}}\partial\underline{t}(\bar{y})+\overline{\delta_{0}}\partial\overline{t}(\bar{y})+ \sum_{j=1}^{m} \delta_j \partial g_j(\bar{y})$.
\end{note}

We will now establish a KKT-type condition. The necessity of the conditions requisite an extra condition, which refer as \emph{Slater's condition}:
\begin{align}\label{crt}
\text{there exists }\bar{y}\in \mathcal{Y} \text{ such that } g_i(\bar{y})<0 ~\text{ for all  }~i=1,2,\ldots,p.   
\end{align}

\begin{thm}{\emph{(KKT-type necessary condition for IOPs).}}\label{kkeff} Let $\mathcal{Y}$ be a nonempty convex subset of $\mathbb{R}^n$. 
Consider the constrained IOP 
\begin{eqnarray}\label{p3}
&\underset{y \in \mathcal{Y}}{\inf}& \textbf{T}(y)\\
&\text{subject to }&g_j(y) \leq 0,~j=1,2,\ldots,p,\nonumber
\end{eqnarray}
where $\textbf{T}: \mathcal{Y} \to I(\mathbb{R})$ be a convex IVF and $g_1,g_2,\ldots,g_p : \mathcal{Y} \to \mathbb{R}$ be real-valued convex functions. Assume that the Slater's condition (\ref{crt}) is satisfied. Then, $\bar{y}$ is a weak efficient solution of (\ref{p3}) if and only if there exist $\delta_j \geq 0,~ j=1,2,\ldots,p$, such that 
\begin{eqnarray}
&&\widehat{\textbf{0}} \in \partial \textbf{T}(\bar{y}) \oplus \sum_{j=1}^{p} \delta_j \partial g_j(\bar{y})\label{p1}\\
&\text{and}&\delta_j g_j(\bar{y})=0,~j=1,2,\ldots,p.\label{p2}
\end{eqnarray} 
\end{thm}

\begin{proof}
Let $\bar{y}$ be a weak efficient solution of (\ref{p3}). Then, from Theorem \ref{fri}, there exist $\tilde{\delta}_{i} \geq 0,$ not all zeros, such that  
\begin{eqnarray}
&&\widehat{\textbf{0}} \in \tilde{\delta}_{0} \odot \partial \textbf{T}(\bar{y}) \oplus \sum_{j=1}^{p} \tilde{\delta}_j \partial g_j(\bar{y})\text{ and }\label{a1}\\
&&\tilde{\delta}_{j} g_j(\bar{y})=0,~j=1,2,\ldots,p.\label{a3}
\end{eqnarray}
If $\tilde{\delta}_{0}\not= 0$, then the result is true by taking $\delta_j=\frac{\tilde{\delta}_{j}}{\delta_0},~j=1,2,\ldots,p$ in (\ref{a1}). Let us assume contrarily that $\tilde{\delta}_{0}= 0$. Then,
\[
\widehat{\textbf{0}} \in  \sum_{j=1}^{p} \tilde{\delta}_j \partial g_j(\bar{y}) \implies \widehat{0} \in  \sum_{j=1}^{p} \tilde{\delta}_j \partial g_j(\bar{y}).
\]
From Theorem 3.78 of \cite{amir}, for a point $x^*$ which satisfies the Slater's condition (\ref{crt}), we have $ \sum_{j=1}^{p} \tilde{\delta}_{j}g_{j}({x}^{*}) \geq 0,$
which controverts the presumption that $\tilde{\delta}_{j} \geq 0$ and $g_j(x^{*})<0$ for some $j$ and all $\tilde{\delta}_{j}$'s are not zero. Thus, $\tilde{\delta}_{0}>0.$\\ 
To prove the converse part, assume that $\bar{y}$ satisfies (\ref{p1}) and (\ref{p2}) for some $\delta_j\geq 0,~j=1,2,\ldots,p.$\\
Define a convex IVF $\textbf{H}$ by $\textbf{H}(y)=\textbf{T}(y)\oplus \sum_{j=1}^{p}\delta_j g_j(y).$ 
Then,
\[\partial\textbf{H}(\bar{y})=\partial\textbf{T}(\bar{y})\oplus \sum_{j=1}^{p}\delta_j \partial g_j(\bar{y}).\]
Since $\widehat{\textbf{0}}\in \partial \textbf{H}(\bar{y})$, $\bar{y}$ is a weak efficient point of $\textbf{H}$. Therefore, by (\ref{a3}) and Theorem \ref{fermets}, 
\begin{equation}\label{uu}
  \textbf{T}(\bar{y})=\textbf{T}(\bar{y})\oplus \sum_{j=1}^{p} \delta_j g_j(\bar{y})=\textbf{H}(\bar{y}).  
\end{equation}
Let $\hat{y}$ be a feasible point of (\ref{p3}). Then, 
\begin{equation}\label{uu1}
\textbf{H}(\hat{y})= \textbf{T}(\hat{y}) \oplus \sum_{j=1}^{p} \delta_j g_j(\hat{y}) \preceq \textbf{T}(\hat{y}).
\end{equation}
From \eqref{uu} and \eqref{uu1}, we obtain $ \textbf{T}(\bar{y}) \preceq \textbf{T}(\hat{y})$, and therefore $\bar{y}$ is a weak efficient solution of (\ref{p3}).
\end{proof}




\section{Application of $gH$-subdifferentials in nonconvex composite optimization models} \label{section4}
In this section, we derive a necessary efficiency condition for nonconvex composite IOPs and a sufficient condition of convex IOPs with the help of $gH$-subdifferentials.

\begin{thm}{\emph{(Efficiency conditions for the composite model).}} \label{non1}
Let $\mathcal{Y}$ be a nonempty convex subset of $\mathbb{R}^n$. Let $\textbf{T} : \mathcal{Y} \rightarrow I(\mathbb{R})\cup \{\boldsymbol{+\infty}\}$ be a proper IVF and let $\textbf{H} : \mathcal{Y} \rightarrow I(\mathbb{R})\cup \{\boldsymbol{+\infty}\}$ be a proper convex IVF such that dom($\textbf{H}$) is a subset of the interior of dom($\textbf{T}$).  
Consider the IOP $(\textbf{P})$:
\begin{eqnarray} \label{0_7}
\underset{y \in \mathcal{Y}}{\inf}~\textbf{T}(y) \oplus \textbf{H}(y).
\end{eqnarray}
If $\bar{y} \in dom(\textbf{H})$ is a weak efficient solution of (\ref{0_7}) 
and $\textbf{T}$ is $gH$-differentiable at $\bar{y}$, then 
    \begin{eqnarray}\label{0_8}
   (-1) \odot\nabla\textbf{T}(\bar{y}) \in \partial \textbf{H}(\bar{y}).
    \end{eqnarray}
The converse is true if $\textbf{T}$ is convex on $\mathcal{Y}$.
\end{thm}

\begin{proof}
Let $\bar{y}, y \in \ $dom($\textbf{H}$) and $\delta \in (0,1)$ such that $y_{\delta} = (1-\delta)\bar{y} + \delta y  \in$ dom($\textbf{H}$). \\
Since $\bar{y}$ is a weak efficient solution of (\ref{0_7}), 
\begin{eqnarray*}
   &&\textbf{P}(\bar{y})~\preceq~\textbf{P}(y_{\delta})\\
   &\implies&  \textbf{T}(\bar{y})\oplus\textbf{H}(\bar{y})~\preceq~\textbf{T}(y_{\delta}) \oplus\textbf{H}(y_{\delta}) \\
   &\implies&\textbf{T}(\bar{y})\oplus\textbf{H}(\bar{y})~\preceq~\textbf{T}((1-\delta)\bar{y} + \delta y) \oplus \textbf{H}( \delta y + (1-\delta)\bar{y} ).
\end{eqnarray*}
Due to convexity of $\textbf{H}$ and (\ref{3_5}) of Lemma \ref{00}, we have
\begin{eqnarray*}\label{equ_19}
   &&\textbf{T}(\bar{y})\oplus\textbf{H}(\bar{y})~\preceq~\textbf{T}((1-\delta)\bar{y} + \delta y) \oplus \delta\odot \textbf{H}(y) \oplus (1-\delta)\odot\textbf{H}(\bar{y})\nonumber\\
   &\implies&\textbf{T}(\bar{y}) \oplus \delta \odot\textbf{H}(\bar{y})~\preceq~\textbf{T}((1-\delta)\bar{y} + \delta y) \oplus \delta\odot \textbf{H}(y)~\text{by Note~\ref{notei}}\\
  &\implies& \frac{1}{\delta}\odot\big(\textbf{T}(\bar{y})\ominus_{gH}\textbf{T}((1-\delta)\bar{y}+\delta y)\big)~\preceq~\textbf{H}(y) \ominus_{gH} \textbf{H}(\bar{y})~\text{by (\ref{3_3}) of Lemma~\ref{00}}\\
  &\implies& (-1)\odot\frac{1}{\delta}\odot\bigg(\textbf{T}((1-\delta)\bar{y} + \delta y)\ominus_{gH}\textbf{T}(\bar{y})\bigg)~\preceq~\textbf{H}(y) \ominus_{gH} \textbf{H}(\bar{y}).
\end{eqnarray*}
 Since $\textbf{T}$ is $gH$-differentiable, as $\delta \to 0^+$, we have
\begin{eqnarray*}
  && (y-\bar{y})^\top \odot ( (-1) \odot \nabla \textbf{T}(\bar{y}) ) \preceq \textbf{H}(y) \ominus_{gH} \textbf{H}(\bar{y})\\
  &\Longrightarrow&(-1) \odot \nabla \textbf{T}(\bar{y}) \in \partial \textbf{H}(\bar{y}).
\end{eqnarray*}
To prove the latter part, we assume that $\textbf{T}$ is convex on $\mathcal{Y}$ and $(-1) \odot \nabla \textbf{T}(\bar{y}) \in \partial \textbf{H}(\bar{y})$. If possible, let $\bar{y}$ is not a weak efficient point of the problem (\ref{0_7}). Then, there exists a point $y{'} \in$ dom($\textbf{H}$) such that $
   \textbf{P}(y{'})~\preceq~\textbf{P}(\bar{y}).$ 
Therefore, for any $\delta, \delta' \in (0,1)$ with $\delta+\delta'=1$, we have
\begin{align*}
   &\delta\odot\textbf{P}(y{'})~\preceq~\delta\odot\textbf{P}(\bar{y})\\
   \text{or, }~&\delta\odot\textbf{P}(y{'})\oplus\delta{'}\odot\textbf{P}(\bar{y})~\preceq~\delta\odot\textbf{P}(\bar{y})\oplus\delta{'}\odot\textbf{P}(\bar{y})\\
   \text{or, }~&\delta\odot\textbf{P}(y{'})\oplus\delta{'}\odot\textbf{P}(\bar{y})~\preceq~(\delta\oplus\delta{'})\odot\textbf{P}(\bar{y})~=~\textbf{P}(\bar{y}).
\end{align*}
Due to convexity of $\textbf{T}$ and $\textbf{H}$ on $\mathcal{Y}$, we have
\begin{eqnarray*}
    &&\textbf{T}(\delta y{'}+\delta{'}\bar{y})\oplus\delta\odot\textbf{H}(y{'})\oplus\delta{'}\odot\textbf{H}(\bar{y})\\
    &\preceq&\delta\odot\textbf{T}(y{'})\oplus\delta{'}\odot\textbf{T}(\bar{y})\oplus\delta\odot\textbf{H}(y){'}\oplus\delta{'}\odot\textbf{H}(\bar{y})\\
    &\preceq&\textbf{T}(\bar{y})\oplus \textbf{H}(\bar{y}).
\end{eqnarray*}
Thus, from (\ref{3_3}) of Lemma \ref{00}, we get
\begin{eqnarray*}
&& \textbf{T}(\bar{y}+\delta(y{'}-\bar{y}))\oplus\delta\odot\textbf{H}(y{'})~\preceq~\textbf{T}(\bar{y})\oplus \delta\odot\textbf{H}(\bar{y})\\
&\text{or,}& \textbf{H}(y{'}) \ominus_{gH}\textbf{H}(\bar{y})~\preceq~\frac{1}{\delta}\odot(\textbf{T}(\bar{y})\ominus_{gH}\textbf{T}(\bar{y}+\delta(y{'}-\bar{y}))\\
 &\text{or,}&\textbf{H}(y{'})\ominus_{gH}\textbf{H}(\bar{y})~\preceq~(-1)\odot\frac{1}{\delta}\odot(\textbf{T}(\bar{y}+\delta(y{'}-\bar{y})).
\end{eqnarray*}
This is a contradiction to the asuumption that  $(-1) \odot \nabla \textbf{T}(\bar{y})\in \partial \textbf{H}(\bar{y})$ for all $y \in \mathcal{Y}$.\\
Hence, $\bar{y}$ is a weak efficient point of $\textbf{T}$.
\end{proof}

\begin{note}\label{no2}
The converse of Theorem \ref{non1} is not true if $\textbf{T}$ is a nonconvex IVF. For instance, consider the IOP
\begin{eqnarray}\label{pka}
\underset{y \in [-5,5]}{\inf}~\textbf{T}(y) \oplus \textbf{H}(y),\end{eqnarray}
where $\textbf{T},\textbf{H}:[-5,5] \to I(\mathbb{R})$ are defined by
\[
\textbf{T}(y) = [2,4] \odot y^{3} \oplus [1,1]~\text{and}~\textbf{H}(y) = [3,3].
\]
Therefore,
\[
\underline{t}(y)=\begin{cases}2y^3+1, & \text{for } y \geq 0 \\
4y^3+1, & \text{for } y < 0,
\end{cases}
~\text{and}~\overline{t}(y)=\begin{cases}
4y^3+1, & \text{for } y \geq 0 \\
2y^3+1,  & \text{for } y < 0.
\end{cases}
\]

\begin{figure} 
\begin{center}	
\begin{tikzpicture}[scale=1.2]
\begin{axis}[domain=-1:1,xmin=-1.2, xmax=1.2, ymin=-3.5, ymax=5.5, axis x line=middle, axis y line=middle, xlabel=$y$, ylabel=$T$, 
xtick = {-1, -0.5, ..., 1}, 
minor xtick = {-0.75, -0.5, ..., 1},
ytick = {-2, 0, ..., 4}, 
minor ytick = {-2, -1, ..., 4}, 
no markers]
\addplot+[name path=A,red]{2*x^3+1};
\addplot+[name path=B,red]{4*x^3+1};
\addplot[mylightblue] fill between[of=A and B];
\draw[color=black]  (41,350) -- (57, 350);

\draw (112,318) node {$O$};

\draw[->]  (41,490) -- (55, 400);
\draw (39,500) node {$\overline{t}$};

\draw[->]  (161,615) -- (175, 525);
\draw (159,635) node {$\overline{t}$};

\draw[->]  (55, 160) -- (41,250);
\draw (59,170) node {$\underline{t}$};

\draw[->]  (205, 430) -- (191,520);
\draw (208, 410) node {$\underline{t}$};

\end{axis}
\end{tikzpicture}
\caption{The IVF $\textbf{T}$ of Note \ref{no2}}\label{fexcd2}
\end{center}
\end{figure}

\noindent	
From Lemma \ref{lc1}, it is clear that the IVF $\textbf{T}$ is not convex as $\underline{t}$ and $\overline{t}$ are not convex (refer to Figure \ref{fexcd2}). 
Note that $\textbf{T}$ and $\textbf{H}$ are $gH$-differentiable and 
\[
\nabla \textbf{T}(y) = [6,12] \odot y^{2}~\text{and}~\nabla \textbf{H}(y) = \textbf{0}.
\]
At $\bar{y}=0$, $\partial \textbf{H}(\bar{y})=\{ \nabla \textbf{H}(\bar{y}) \}=\{\textbf{0}\}$, and $\nabla \textbf{T}(\bar{y}) = \textbf{0}$. 
Hence, $(-1) \odot \nabla \textbf{T}({0}) \in \partial \textbf{H}({\bar{y}})$. However, $\bar{y}$ is not an efficient point of the IOP (\ref{pka}) (refer to Figure \ref{fexcd2}). 
\end{note}

In the below Examples \ref{empl1} and \ref{empl2}, we exemplify Theorem \ref{non1} by considering some special cases for $\textbf{H}$.

\begin{dfn} {(Indicator function for IVF)}. 
Let $\mathcal{Y}$ be a subset of $\mathbb{R}^{n}.$ Then, the indicator function $ \boldsymbol{\delta}_{\mathcal{Y}}:\mathbb{R}^{n}\rightarrow \overline{I(\mathbb{R})}$ at a point $y$, is defined as 
\[
 \boldsymbol{\delta}_{\mathcal{Y}}(y) =
\begin{cases}
 \textbf{0}, & \text{if}~ y \in \mathcal{Y}\\
  \boldsymbol{+\infty}, & \text{if}~ y \notin \mathcal{Y}.
\end{cases}
\]
\end{dfn}

\begin{example}{(Convex constrained nonconvex programming problem for IVF).}\label{empl1} Let $\mathcal{Y}$ be a nonempty convex subset of $\mathbb{R}^n$. Let $\textbf{T} : \mathcal{Y} \rightarrow I(\mathbb{R})\cup \{\boldsymbol{+\infty}\} $ be a nonconvex $gH$-differentiable IVF and $\textbf{H}:\mathcal{Y} \rightarrow \overline{I(\mathbb{R})}$ be a convex IVF which is defined by $\textbf{H}(y) = \boldsymbol{\delta}_{\mathcal{Y}}(y)$.
Consider the IOP
\begin{eqnarray}\label{0_9}
&& \underset{y \in \mathcal{Y}}{\inf}~\textbf{T}(y)\oplus \textbf{H}(y).
\end{eqnarray}
If $\bar{y} \in \mathcal{Y}$ is a weak efficient point of (\ref{0_9}), then from Theorem \ref{non1}, we have 
\[
(-1) \odot \nabla \textbf{T}(\bar{y}) \in  \partial \textbf{H}(\bar{y}) \implies(-1) \odot \nabla \textbf{T}(\bar{y}) \in \partial \boldsymbol{\delta}_{\mathcal{Y}}(\bar{y}).
\]
We observe that if $\textbf{G}\in I(\mathbb{R})$ is a subgradient of $\boldsymbol{\delta}_{\mathcal{Y}}$ at $\bar{y}\in \mathcal{Y},$ then for all $y\in \mathcal{Y}$
\begin{eqnarray*}
&&(y-\bar{y})^\top\odot\textbf{G} \preceq \boldsymbol{\delta}_{\mathcal{Y}}(y)\ominus_{gH} \boldsymbol{\delta}_{\mathcal{Y}}(\bar{y})\\
&\implies& (y-\bar{y})^\top\odot\textbf{G} \preceq \textbf{0}.
\end{eqnarray*}
Therefore, we have
\[
 (-1) \odot (y-\bar{y})^\top\odot \nabla \textbf{T}(\bar{y}) \preceq \textbf{0}.
\]
\end{example}
%
%
\begin{example}\label{empl2} Let the generic element of $\mathbb{R}^n$ be $y = (y_1, y_2, \ldots, y_n)^\top$ and $p \in \{1, 2, \ldots, n\}$.  Let $\textbf{T} : \mathbb{R}^n \rightarrow I(\mathbb{R}) \cup \{+\boldsymbol{\infty}\}$ be a $gH$-differentiable IVF and $\textbf{H} : \mathbb{R}^n \rightarrow  I(\mathbb{R}) \cup \{\boldsymbol{+\infty}\}$ be a convex IVF. Consider the IOP
\begin{eqnarray}\label{1_0}
&& \underset{y \in \mathcal{Y}}{\inf}~\textbf{T}(y)\oplus  \textbf{H}(y)
\end{eqnarray}
where $\textbf{0} \preceq \textbf{C}$ and $\textbf{H}(y) = \textbf{C} \odot |y_p |$. \\
If $\bar{y} \in \mathcal{Y}$ is a weak efficient point of (\ref{1_0}), then by  Theorem \ref{non1}, we have 
\[
(-1) \odot \nabla \textbf{T}(\bar{y}) \in \partial \textbf{H}(\bar{y}).
\]
From \cite{Ghosh2020lasso}, the $gH$-subdifferential set of $\textbf{H}$ at any  $\bar{y}=(\bar{y}_1,\bar{y}_2,\ldots,\bar{y}_{p-1},0,\bar{y}_{p+1},\ldots,\bar{y}_n)^\top$ in the plane $y_p = 0$ is given by 
\[
\{(\textbf{G}_1,\textbf{G}_2,\ldots,\textbf{G}_n)^\top \in I(\mathbb{R})^n : (-1)\odot \textbf{C} \preceq \textbf{G}_j \preceq \textbf{C}, \text{ for all } j = 1, 2, \ldots, n\}.
\] Therefore,
\[
	\nabla\textbf{T}(\bar{y})=(D_1\textbf{T}(\bar{y}),D_2\textbf{T}(\bar{y}),\ldots,D_n\textbf{T}(\bar{y}))^\top,\]
is given by 
\begin{eqnarray}\label{crst} 
{D}_{j}{\textbf{T}}(\bar{y}) =
\begin{cases} \label{}
(-1)\odot \textbf{C} & \text{if }\bar{y}_{j} < 0,\\
~ \textbf{C} & \text{if }\bar{y}_{j} > 0,\\
\textbf{G}_i \in I(\mathbb{R}): (-1)\odot \textbf{C} \preceq \textbf{G}_j \preceq \textbf{C}  & \text{if }\bar{y}_{j} = 0.
\end{cases}
\end{eqnarray}
for each $j = 1, 2, \ldots, n$. Thus, (\ref{crst}) is a necessary condition for $\bar{y}$ to be a weak efficient point of the IOP (\ref{1_0}). 
\end{example}
%
%
%

\section{Conclusion and future directions}\label{section5}
In this article, three major results on IVFs and IOPs have been derived---$gH$-directional derivative of the maximum of IVFs (Theorem \ref{dd_max}), Fritz-John-type necessary efficiency condition (Theorem \ref{fri}), and KKT-type necessary and sufficient efficiency condition for IOPs (Theorem \ref{kkeff}). To derive these results, we have defined and analyzed the concepts of infimum (Definition \ref{inf}), supremum (Definition \ref{sup}), closedness (Definition \ref{closedset}), boundedness (Definition \ref{boundedset}), and convex hull (Definition \ref{cohull}) in $I(\mathbb{R})$; also, we have derived some properties related to these concepts. One can trivially notice that in the degenerate case, Definitions \ref{inf} and \ref{sup} reduces to the respective conventional definition for the real-valued functions (see \cite{amir,dhara}). By using the proposed  calculus for IVFs, a characterization of the efficient solutions of a nonconvex composite model with IVFs (Theorem \ref{non1}) has been derived. \\

In connection with the proposed research, few future directions are as follows:
\begin{itemize}[$\bullet$]
\item  A $gH$-subgradient method for unconstrained nonconvex nonsmooth IOPs can be attempted to be developed.
\item  The convexity properties of the optimal value function can be analyzed under inequality and equality constraints.
\item  We may attempt to derive a sufficient conditions which ensures the existence of a subgradient of the value function associated with convex IOPs. Towards this direction, one may consider the following IOP:  
\begin{align*}
\text{min}~&~\textbf{F}(y)\\
\text{subject to}~&~\textbf{G}_{j}(y) \preceq \textbf0,~j=1,2,\ldots,m\\
~&~\textbf{H}_{l}(y)= \textbf0,~l=1,2,\ldots,p \\
~&~y \in \mathcal{Y},
\end{align*}
where $\textbf{F},\textbf{G}_{j},\textbf{H}_l:\mathcal{Y} \to I(\mathbb{R})\cup \{\boldsymbol{+\infty}\}$ are  extended convex IVFs. The value function with above problem is the IVF $\textbf{V} : \mathbb{R}^m\times \mathbb{R}^p\to \overline{I(\mathbb{R})} $, which may be defined as follows: 
\[
\textbf{V}(u,t)=\underset{y \in \mathcal{Y}}{\text{inf}} \{\textbf{F}(y):\textbf{G}_{j}(y)\preceq [u_j,u_j],~~\textbf{H}_{l}(y)= [t_l,t_l] \text{ for all } j \text{ and } l\}.\]
For such a value function, one may attempt to develop the subdifferential set with the help of the proposed calculus rules of this articles and try to analyze its sensitivity with variational inequality constraints.
\end{itemize}



%

\appendix

\section{Proof of Lemma \ref{00}} \label{appendix_inq}
\begin{proof}
 Let $\textbf{X}=[\underline{x},\overline{x}],~ \textbf{Y}=[\underline{y},\overline{y}],~ \textbf{Z}=[\underline{z},\overline{z}],~\text{and}~\textbf{W}=[\underline{w},\overline{w}]$.
\begin{enumerate}[(i)]



  \item Since $\textbf{X} \oplus \textbf{Y} \preceq \textbf{Z} \oplus \textbf{W}$,
  \begin{align}\label{ineq_11}
  & \underline{x}+\underline{y}~\preceq~\underline{z}+\underline{w}~\text{and}~\overline{x}+\overline{y}~\preceq~\overline{z}+\overline{w}\nonumber\\
  \implies& \underline{x}-\underline{z}~\preceq~\underline{w}-\underline{y}~\text{and}~\overline{x}-\overline{z}~\preceq~\overline{s}-\overline{y}\nonumber\\
  \implies& \text{min}\{\underline{x}-\underline{z},\overline{x}-\overline{z}\}~\preceq~\text{min}\{\underline{w}-\underline{y},\overline{w}-\overline{y}\},~\text{and} \nonumber\\
  &\text{max}\{\underline{x}-\underline{z},\overline{x}-\overline{z}\}~\preceq~\text{max}\{\underline{w}-\underline{y},\overline{w}-\overline{s}\}. 
  \end{align}
  Hence, $\textbf{X} \ominus_{gH} \textbf{Z}~\preceq~\textbf{W} \ominus_{gH} \textbf{Y}$.


\item 
Since $\textbf{X}\preceq\textbf{Y}$ and $\textbf{Y}\preceq\textbf{Z}$, we have $\underline{x}\leq\underline{y},~ \overline{x}\leq\overline{y}$ and $\underline{y}\leq\underline{z},~ \overline{y}\leq\overline{z}$.\\
This implies $\underline{x}\leq\underline{z}~\text{and}~\overline{x}\leq\overline{z}, \text{ i.e., }\textbf{X}\preceq\textbf{Z}.$

\end{enumerate}
\end{proof}


\section{Proof of Lemma \ref{infsup}} \label{appendix_insu}
\begin{proof}
Let $\textbf{B}_{1}=\{\textbf{T}(\widehat{\textbf{X}}):\widehat{\textbf{X}}\in\textbf{S}_{1}\}$, $\textbf{B}_{2}=\{\textbf{T}(\widehat{\textbf{X}}):\widehat{\textbf{X}}\in\textbf{S}_{2}\}$, and $\textbf{B}=\{\textbf{T}(\widehat{\textbf{X}}):\widehat{\textbf{X}}\in\textbf{S}\}$.\\
\begin{enumerate}[(i)]
\item From Definition \ref{inf}, we have
   \begin{eqnarray*}
    \underset{\textbf{S}_{2}}{\inf}~\textbf{T}= \text{inf}\{\textbf{T}(\widehat{\textbf{X}}):\widehat{\textbf{X}}\in\textbf{S}_{2}\}=\text{inf}~\textbf{B}_{2}=\overline{\textbf{M}},\text{ where }\overline{\textbf{M}}\in I(\mathbb{R}).
    \end{eqnarray*}
    Since $\overline{\textbf{M}}$ is a lower bound of $\textbf{B}_{2}$ and $\textbf{S}_{1} \subseteq \textbf{S}_{2}$, $\overline{\textbf{M}}$ is also a lower bound of $\textbf{B}_{1}$.  This implies that
    \begin{eqnarray*}
     \overline{\textbf{M}}\preceq\underset{\textbf{S}_{1}}{\inf}~\textbf{T}
    \implies \underset{\textbf{S}_{2}}{\inf}~\textbf{T}\preceq\underset{\textbf{S}_{1}}{\inf}~\textbf{T}.
    \end{eqnarray*}
    

\item 
From Definition \ref{sup}, we have
   \begin{eqnarray*}
   \underset{\textbf{S}_{2}}{\sup}~\textbf{T}= \text{sup}\{\textbf{T}(\widehat{\textbf{X}}):\widehat{\textbf{X}}\in\textbf{S}_{2}\}=\text{inf}~\textbf{B}_{2}=\overline{\textbf{N}},\text{ where }\overline{\textbf{N}}\in I(\mathbb{R}).
    \end{eqnarray*}
    Since $\overline{\textbf{N}}$ is an upper bound of $\textbf{B}_{2}$ and $\textbf{S}_{1} \subseteq \textbf{S}_{2}$, therefore $\overline{\textbf{N}}$ is also an upper bound of $\textbf{B}_{1}$. This implies that
    \begin{eqnarray*}
    &&\underset{\textbf{S}_{1}}{\sup}~\textbf{T}\preceq\overline{\textbf{M}}
    \implies \underset{\textbf{S}_{1}}{\sup}\textbf{T}\preceq\underset{\textbf{S}_{2}}{\sup}~\textbf{T}.
    \end{eqnarray*}


\item 
From Definition \ref{inf}, we have
  \begin{eqnarray*}
  \underset{\textbf{S}}{\inf}~\textbf{T}=\text{inf}\{\textbf{T}(\widehat{\textbf{X}}):\widehat{\textbf{X}}\in\textbf{S}\}=\text{inf}~\textbf{B}=\overline{\textbf{M}}.
  \end{eqnarray*}
   Thus, for each $\widehat{\textbf{X}}\in\textbf{S}$ and for every $\delta \geq$ 0, we get 
      \begin{eqnarray}\label{equa_b1}
      \overline{\textbf{M}}\preceq\textbf{T}(\widehat{\textbf{X}})
      \implies\delta\odot\overline{\textbf{M}}\preceq\delta\odot\textbf{T}(\widehat{\textbf{X}}) 
      \implies \delta\odot\overline{\textbf{M}}\preceq\underset{\textbf{S}}{\inf}~(\delta\odot\textbf{T}).
      \end{eqnarray}
    Since $\overline{\textbf{M}}$ is an infimum of $\textbf{B}$, for given $\epsilon >0$ and $\delta > 0$, we have 
      \begin{eqnarray*}
      &&\textbf{T}(\widehat{\textbf{X}}_{1})\prec\overline{\textbf{M}}\oplus\left[\tfrac{\epsilon}{\delta},\tfrac{\epsilon}{\delta}\right]~\text{for some~}\widehat{\textbf{X}}_{1}\in\textbf{S}\\
      \implies&&\delta\odot\textbf{T}(\widehat{\textbf{X}}_{1})\prec\delta\odot\left(\overline{\textbf{M}}\oplus\left[\tfrac{\epsilon}{\delta},\tfrac{\epsilon}{\delta}\right]\right)\\
        \implies&&\delta\odot\textbf{T}(\widehat{\textbf{X}}_{1})\prec(\delta\odot\overline{\textbf{M}})\oplus[\epsilon,\epsilon].
       \end{eqnarray*}
       Due to arbitrariness of $\epsilon$, any interval $\textbf{C}\in I(\mathbb{R})$ such that $\delta\odot\overline{\textbf{M}}\prec\textbf{C}$ cannot be a lower bound of $\delta\odot\textbf{B}$. Therefore,
        \begin{eqnarray} \label{equa_b2}
        \underset{\textbf{S}}{\text{inf}}(\delta\odot\textbf{T})\preceq\delta\odot\overline{\textbf{M}}.
         \end{eqnarray}
         From (\ref{equa_b1}) and (\ref{equa_b2}), we obtain
 $\underset{\textbf{S}}{\inf}~(\delta\odot\textbf{T})~=~\delta\odot\underset{\textbf{S}}{\inf}~\textbf{T}.$


\item 
From Definition \ref{sup}, we have
   \begin{eqnarray*}
  \underset{\textbf{S}}{\sup}~\textbf{T}=\text{sup}\{\textbf{T}(\widehat{\textbf{X}}):\widehat{\textbf{X}}\in\textbf{S}\}=\text{sup}~\textbf{B}= \overline{\textbf{N}}.
  \end{eqnarray*}
     Thus, for each $\widehat{\textbf{X}}\in\textbf{S}$ and $\delta \geq 0$, we get
      \begin{eqnarray}\label{eqa_1}
      \textbf{T}(\widehat{\textbf{X}})\preceq\overline{\textbf{N}}
      &\implies&\delta\odot\textbf{T}(\widehat{\textbf{X}})\preceq\delta\odot\overline{\textbf{N}}
     \implies \underset{\textbf{S}}{\sup}~(\delta\odot\textbf{T})\preceq\delta\odot\overline{\textbf{N}}.
      \end{eqnarray}
      Since $\overline{\textbf{N}}$ is a supremum of $\textbf{B}$, for given $\epsilon >0$ and $\delta >0$, we have
      \begin{eqnarray*}
      &&\overline{\textbf{N}}\ominus_{gH}\left[\tfrac{\epsilon}{\delta},\tfrac{\epsilon}{\delta}\right]\prec\textbf{F}(\widehat{\textbf{X}}_{1})~\text{for some~}\widehat{\textbf{X}}_{1}\in\textbf{S}\\
      \implies&&\delta\odot\left(\overline{\textbf{N}}\ominus_{gH}\left[\tfrac{\epsilon}{\delta},\tfrac{\epsilon}{\delta}\right]\right)\prec\delta\odot\textbf{T}(\widehat{\textbf{X}}_{1})\\
        \implies&&\delta\odot\overline{\textbf{N}}\ominus_{gH}[\epsilon,\epsilon]\prec\delta\odot\textbf{T}(\widehat{\textbf{X}}_{1}).
       \end{eqnarray*}
       Due to arbitrariness of $\epsilon$, any interval $\textbf{C}\in I(\mathbb{R})$ such that $\textbf{C}\prec\delta\odot\overline{\textbf{N}}$ cannot be an upper bound of $\delta\odot\textbf{B}$. Therefore,
        \begin{eqnarray} \label{eqa_2}
        \delta\odot\overline{\textbf{N}}\preceq\underset{\textbf{S}}{\text{sup}}(\delta\odot\textbf{T}).
         \end{eqnarray}
         In view of (\ref{eqa_1}) and (\ref{eqa_2}), we obtain $\underset{\textbf{S}}{\sup}~(\delta\odot\textbf{T})~=~\delta\odot\underset{\textbf{S}}{\sup}~\textbf{T}.$
  \end{enumerate}

\end{proof}


%
\section{Proof of Lemma \ref{supinf}} \label{appendix_suprinfm}
\begin{proof}
\begin{enumerate}[(i)]
\item From Definitions \ref{lbound} and \ref{inf}, for each $\widehat{\textbf{X}}\in\textbf{S}$ we have
\begin{eqnarray*}
   &&\underset{\textbf{S}}{\text{inf}}~\textbf{T}_1\preceq\textbf{T}_1(\widehat{\textbf{X}})
    ~\text{and}~\underset{\textbf{S}}{\text{inf}}~\textbf{T}_2\preceq\textbf{T}_2(\widehat{\textbf{X}})\\
 \implies&&\underset{\textbf{S}}{\text{inf}}~\textbf{T}_1\oplus\underset{\textbf{S}}{\text{inf}}~\textbf{T}_2\preceq\textbf{T}_1(\widehat{\textbf{X}})\oplus\textbf{T}_2(\widehat{\textbf{X}})~\text{from Lemma 2.5 of \cite{gourav2020}}.
   \end{eqnarray*}
   Since $\underset{\textbf{S}}{\inf}~(\textbf{T}_1 \oplus \textbf{T}_2)$ is the infimum of $\textbf{T}_1(\widehat{\textbf{X}})\oplus\textbf{T}_2(\widehat{\textbf{X}})~~\text{for each}~\widehat{\textbf{X}}\in\textbf{S}$,
   \[\underset{\textbf{S}}{\text{inf}}~\textbf{T}_1\oplus\underset{\textbf{S}}{\text{inf}}~\textbf{T}_2\preceq\underset{\textbf{S}}{\inf}~(\textbf{T}_1 \oplus \textbf{T}_2).\]
   

   \item From Definitions \ref{ubound} and \ref{sup}, for each $\widehat{\textbf{X}}\in\textbf{S}$ we have
   \begin{eqnarray*}
   &&\textbf{T}_1(\widehat{\textbf{X}})\preceq\underset{\textbf{S}}{\text{sup}}~\textbf{T}_1~
    \text{and}~\textbf{T}_2(\widehat{\textbf{X}})\preceq\underset{\textbf{S}}{\text{sup}}~\textbf{T}\\
    \implies&&\textbf{T}_1(\widehat{\textbf{X}})\oplus\textbf{T}_2(\widehat{\textbf{X}})\preceq\underset{\textbf{S}}{\text{sup}}~\textbf{T}_1\oplus\underset{\textbf{S}}{\text{sup}}~\textbf{T}_2~\text{from Lemma 2.5 of \cite{gourav2020}}.
   \end{eqnarray*}
   Since $\underset{\textbf{S}}{\sup}~(\textbf{T}_1 \oplus \textbf{T}_2)$ is the supremum of $\textbf{T}_1(\widehat{\textbf{X}})\oplus\textbf{F}_2(\widehat{\textbf{X}})~~\text{for each}~\widehat{\textbf{X}}\in\textbf{S}$,  we have
   \[\underset{\textbf{S}}{\sup}~(\textbf{T}_1 \oplus \textbf{T}_2)\preceq\underset{\textbf{S}}{\text{sup}}~\textbf{T}_1\oplus\underset{\textbf{S}}{\text{sup}}~\textbf{T}_2.\]

\end{enumerate}
\end{proof}


\section*{Acknowledgement}
Radko Mesiar acknowledges the financial support of the grant APVV-18-0052 of the Slovak Research and Development Agency and
the grant Palacky University Olomouc IGAPrF2021.

\end{document}